\documentclass{amsart}

\usepackage{amssymb}
\usepackage{overpic}
\usepackage{graphics}
\usepackage{epsfig}
\usepackage{amsmath}
\usepackage{amsfonts}
\usepackage{paralist}
\usepackage{setspace}

\newtheorem{theorem}{Theorem}[section]
\newtheorem{corollary}{Corollary}

\newtheorem{lemma}[theorem]{Lemma}
\newtheorem{proposition}{Proposition}

\newtheorem{example}{Example}
\theoremstyle{definition}
\newtheorem{definition}[theorem]{Definition}
\newtheorem{remark}{Remark}

\begin{document}

\title[Isospectral Transformations of Dynamical Networks]{Isospectral Compression and Other Useful Isospectral Transformations of Dynamical Networks}

\author{L. A. Bunimovich$^1$}

\author{B. Z. Webb$^2$}

\keywords{Dynamical Network, Graph Transformations, Spectral Equivalence, Global Stability}

\subjclass[2000]{05C50, 15A18, 37C75}

\maketitle

\begin{center}
$^{1}$ \smaller{ABC Math Program and School of Mathematics, Georgia Institute of Technology, 686 Cherry Street, Atlanta, GA 30332, USA}\\
$^{2}$ Department of Mathematics, 308 TMCB, Brigham Young University, Provo, UT 84602, USA\\
E-mail: bunimovih@math.gatech.edu and bwebb@math.byu.edu
\end{center}

\begin{abstract}
It is common knowledge that a key dynamical characteristic of a network is its spectrum (the collection of all eigenvalues of the network's weighted adjacency matrix). In \cite{BW10} we demonstrated that it is possible to reduce a network, considered as a graph, to a smaller network with fewer vertices and edges while preserving the spectrum (or spectral information) of the original network. This procedure allows for the introduction of new equivalence relations between networks, where two networks are spectrally equivalent if they can be reduced to the same network. Additionally, using this theory it is possible to establish whether a network, modeled as a dynamical system, has a globally attracting fixed point (is strongly synchronizing). In this paper we further develop this theory of isospectral network transformations and demonstrate that our procedures are applicable to families of parameterized networks and networks of arbitrary size.
\end{abstract}

\vspace{0.1in}

\textbf{Real networks are often very large with a complicated structure (topology). It is therefore tempting to find ways of reducing (or compressing) them. However, when a network is compressed it is possible for some of the important network properties to be lost.}

\textbf{As most real networks are dynamic a natural goal, when reducing such systems, is to preserve their dynamical characteristics. Arguably, the most important characteristic of a dynamical system is its spectrum. It may seem ``obvious" though, that a reduced (lower-dimensional) network should necessarily have fewer eigenvalues than the initial unreduced system. If such were true then no such procedure could exist.}

\textbf{However, in \cite{BW10} it was shown that this is in fact possible and even easily implemented for the analysis of real networks. Moreover, this procedure allows those working with real networks the freedom to design any criteria for compressing a network that they see fit. For instance, one could remove all vertices (edges) with minimal centrality, in/outer degree, etc. Any general rule that determines a core subnetwork, presumably chosen by an expert (biologist, engineer, etc), is possible.}

\textbf{The present paper further develops this procedure presenting some new results and numerous examples. To make the paper self-contained and accessible to non-mathematicians we present all definitions and results with examples but omit technical proofs if they can be found elsewhere.}

\vspace{0.1in}

\section{introduction}
The study of networks in nature, science, and technology is currently one of the most active areas of research. A good deal of this research however is focused on the structural or static features of these networks \cite{Albert02,Dorogovtsev03,Faloutsos99,Newman06,Porter09,Strogatz03,Watts99} although most real networks are inherently dynamic. That is, each network element has an associated state that changes with time  (e.g. neural networks, metabolic networks, etc.).

Moreover, each such network is characterized by the following features. First, the network elements are themselves dynamical systems that evolve according to their own intrinsic (internal) dynamics when isolated from the other network elements. Second, these network elements interact with one another. This interaction between network elements ensures that these elements do in fact form a network.

The internal dynamics and interactions of a network have already been studied together for quite some time. A popular example are reaction-diffusion equations where the reaction corresponds to the internal dynamics and the diffusion to the interactions of these systems (see \cite{Chazottes05} for instance).

What differentiates the current study of dynamical systems from those that have been done previously is the focus on the graph structure or \emph{topology} of the network. Using this approach one ``freezes" the network dynamics and studies the network's structure of interactions (or graph of interactions). These investigations have lead to numerous discoveries concerning the structure of networks. However, the dynamic characteristics of networks have received far less attention.

To investigate the dynamic properties of networks a general approach was introduced in \cite{Afriamovich07}. The major idea in this work is that a network's dynamics can be analyzed in terms of three key features; (i) the internal (local) dynamics of the network elements, (ii) the interactions between the network elements, and (iii) the topology or structure of the graph of interactions of the network.

Features (i) and (ii) of a network are described by dynamical systems and can therefore be studied using standard methods in the theory of differential equations and dynamical systems (again reaction-diffusion systems serve as a standard example). However, the third feature of a network, its topology, is static and must therefore be studied by other means.

It is worth noting that there is a well developed theory for dynamical networks with a lattice type structure of interactions \cite{Chazottes05}. However, the large majority of real networks do not have a regular structure and therefore require another theory to describe their dynamics.

The approach developed in \cite{Afriamovich07} is based on the idea that the third feature of a dynamical network, its graph structure, can also be described as a dynamical system. This notion has allowed for the development of a fairly general  theory of dynamical networks. However, we note here that these studies do not consider rewirings of dynamical networks. The reason being is that rewiring a network makes its law of evolution nonconstant (i.e. time dependent) in which case the network can no longer be considered as an autonomous dynamical system. Instead, to study rewirings one must consider a dynamical system in a certain space of dynamical networks, which is another problem.

When dealing with real networks one is often confronted with the fact that the dynamics of the individual network elements are unknown. Additionally, the interactions (and interaction strengths) between network elements may also be unknown. Even in such circumstances it is tempting to develop methods to study the dynamics of such systems. But this begs the question: If little about the network is known what can be deduced?

In fact, we do know something. That is, we know that the network is dynamic. Hence, one may try various transforms of the network's topology that preserve important aspects of the network's dynamics. Here we consider one of the most fundamental characteristics of a dynamical system, its spectrum.

In many situations only the network structure is known i.e. which network elements interact and which do not. This structure can therefore be modeled by a graph of these interactions or equivalently by the adjacency matrix of this graph. If the numerical strengths of these interactions are known then we have a weighted graph of interactions and a corresponding weighted adjacency matrix.

Because of the often complicated structure of a network it is tempting to find ways of simplifying this structure while preserving some fundamental property or feature of the network. For instance, it would be nice if we were able to reduce a network onto some subset of its vertices (or edges) while maintaining the spectrum of the network's adjacency matrix. However, reducing a network in this way may seem impossible since larger matrices have a larger spectrum (number of eigenvalues).

In fact, it is possible to do just this. The procedure that allows for such reductions is called an \emph{isospectral graph reduction}. This procedure which was developed in \cite{BW10} was moreover shown to have the following key features:
\begin{enumerate}
 \item \textbf{Easily Implemented}: The procedure of \emph{isospectral graph reduction} is well defined and easily implemented. In particular, it is not necessary to develop sophisticated software to carry out this procedure as it is both simple and straightforward.
 \item \textbf{Flexible}: The procedure is very flexible allowing experimentalists the possibility of reducing a network over any set of network elements (vertices) he or she chooses.
 \item \textbf{Unique Reductions}: After reducing a network it is possible to further reduce the network to an even smaller network. Moreover, the resulting smaller network does not depend on the intermediate choice of vertices but only the final collection (subset) of vertices.
 \item \textbf{New Equivalence Relations}: This procedure introduces new equivalence relations between networks where two networks (graphs) are \emph{spectrally equivalent} if they can be reduced to the same network (graph). This in turn allows for the introduction of new relations between various networks whose topologies look quite different but share similar dynamics.
 \item \textbf{Other Network Transforms}: It is also possible transform a network in a variety of other ways while maintaining the spectrum of the network. This includes isospectral expansions of a network which in turn allow for improved estimates of the network's dynamic stability when compared to other methods.
\end{enumerate}

To make the paper self-contained and accessible to the non-mathematician, who are dealing with real networks, we give a detailed explanation of an \emph{isospectral graph reduction} and related transformations as well as their important features. These key features are illustrated by numerous examples. Proofs of all mathematical statements not found in this paper can be found in \cite{BW10}.

This procedure, although developed as a practical tool for dealing with real networks, has already been shown to be an effective means for obtaining new advances in the classical problem of estimating the eigenvalues of a square matrix \cite{BW09}. Specifically, the eigenvalues estimates associated with Gershgorin and others \cite{Gershgorin31,Varga09} improve as the graph associated with the matrix is reduced.

In the present paper we strengthen these previous results by applying isospectral network transformations to networks of arbitrary size and to families parameter dependent networks. Besides this we use a powerful new concept of a complete structural set to strengthen and clarify some previously obtained results.

Lastly, we note that although our procedure deals with matrices as is therefore linear in nature, it in fact does not have this restriction and is designed for the analysis to nonlinear dynamical systems.

\section{Network as Graphs}

To each network there is an associated weighted directed graph $G=(V,E,\omega)$ called the network's \emph{graph of interactions}. The \emph{vertex set} $V$ represents the elements of the network and the \emph{edge set} $E$ the interaction between these elements. For $V=\{v_1,\dots,v_n\}$ we let $e_{ij}$ denote the edge from vertex $v_i$ to $v_j$. The edge $e_{ij}$ is an element of $E$ if the $i$th network element interacts with (or influences) the $j$th network element. The function $\omega$ gives the edge weights of $G$ where $\omega(e_{ij})$, or the edge weight of $e_{ij}$, corresponds to the strength of the interaction between the $i$th and $j$th elements of a network. We adopt the standard convention that each edge of the weighted graph $G$ has a nonzero weight.

\subsection{Isospectral Graph Reductions}
In this section we formally describe the isospectral reduction process of a graph where each graph is considered to be the graph of some dynamical network. Specifically, we consider those graphs that are weighted, directed, with edge weights in the set $\mathbb{W}[\lambda]$ (defined below). Such graphs form the class $\mathbb{G}$.

\begin{remark}
The class of graphs $\mathbb{G}$ is very general as it contains, for instance, the class of undirected graphs with numerical weighs.
\end{remark}

Let $\mathbb{C}[\lambda]$ be the set of polynomials in the complex variable $\lambda$ with complex coefficients. We denote by $\mathbb{W}[\lambda]$ the set of rational functions of the form $p/q$ where $p,q\in\mathbb{C}[\lambda]$, and $q\neq 0$. As we will be dealing with the eigenvalues of matrices which are sets that include multiplicities, we mention the following.

The element $\alpha$ of the set $A$ that includes multiplicities has \textit{multiplicity} $m$ if there are $m$ elements of $A$ equal to $\alpha$. If $\alpha\in A$ with multiplicity $m$ and $\alpha\in B$ with multiplicity $n$ then\\
(i) the \textit{union} $A\cup B$ is the set in which $\alpha$ has multiplicity $m+n$; and\\
(ii) the \textit{difference} $A-B$ is the set in which $\alpha$ has multiplicity $m-n$ if $m-n>0$ and where $\alpha\notin A-B$ otherwise.

\begin{definition}\label{def1.1}
Let $\mathbb{W}[\lambda]^{n\times n}$ denote the set of $n\times n$ matrices with entries in $\mathbb{W}[\lambda]$. For a matrix $A\in\mathbb{W}[\lambda]^{n\times n}$ suppose that $\det(A-\lambda I)=p(\lambda)/q(\lambda)$ where $p(\lambda),q(\lambda)\in\mathbb{C}[\lambda]$. Define the sets
$$P=\{\lambda\in\mathbb{C}:p(\lambda)=0\} \ \ \text{and} \ \ Q=\{\lambda\in\mathbb{C}:q(\lambda)=0\}$$
where these sets includes multiplicities. We call the sets $$\sigma\big(A\big)=P-Q \ \text{and} \ \sigma^{-1}\big(A\big)=Q-P$$ the \textit{spectrum} (or set of \textit{eigenvalues}) of $A$ and the \textit{inverse spectrum} of $A$ respectively.
\end{definition}

We note that both the \textit{spectrum} of $A$ and the \textit{inverse spectrum} of $A$ are sets that include multiplicities where the multiplicity of each element of $\sigma(A)$ and $\sigma^{-1}(A)$ depend on $P$ and $Q$ according to $(ii)$.

It is worth mentioning that the representation of $p(\lambda)/q(\lambda)\in\mathbb{W}[\lambda]$ is not unique. That is, $p(\lambda)/q(\lambda)$ is equivalent to $r(\lambda)/s(\lambda)$ for $p(\lambda),q(\lambda),r(\lambda),s(\lambda)\in\mathbb{C}[\lambda]$ if $p(\lambda)s(\lambda)=q(\lambda)r(\lambda)$. However, it can be shown that the spectrum and inverse spectrum of a matrix $A\in\mathbb{W}[\lambda]^{n\times n}$ are well defined, i.e. do not depend on the particular representation of $\det(A-\lambda I)$.

Suppose $G=(V,E,\omega)$ with vertex set $V=\{v_1,\dots,v_n\}$. We define the matrix $M(G)\in\mathbb{W}[\lambda]^{n\times n}$ entrywise by $$M(G)_{ij}=\omega(e_{ij}).$$
The matrix $M(G)$ is called the \textit{weighted adjacency matrix} of $G$. For the graph $G$ we denote by $\sigma(G)$ and $\sigma^{-1}(G)$ the \emph{spectrum} of and \emph{inverse spectrum} of $M(G)$ respectively.

The study of digraph spectra has and continues to be an extremely active area of research \cite{Bru2010}. However, the approach developed in \cite{BW10}, and described here, differs from previous studies in that we consider graphs with weights in $\mathbb{W}[\lambda]$.

The reason we consider graphs with weights in $\mathbb{W}[\lambda]$ (or matrices with entries in $\mathbb{W}[\lambda]$) is that we wish to reduce the size of the graph (matrix) while maintaining its spectrum. We note that this is not possible if we restrict ourselves to matrices with numerical weights (entries) since a matrix $A\in\mathbb{C}^{n\times n}$ has exactly $n$ eigenvalues including multiplicities. However, a matrix $A\in\mathbb{W}[\lambda]^{n\times n}$ may have more than $n$ eigenvalues in its spectrum. This is demonstrated in example 2.

A \textit{path} $P$ in  the graph $G=(V,E,\omega)$ is an ordered sequence of distinct vertices $v_1,\dots,v_m\in V$ such that $e_{i,i+1}\in E$ for $1\leq i\leq m-1$. We call the vertices $v_2,\dots,v_{m-1}$ of $P$ the \textit{interior} vertices of $P$. If the vertices $v_1$ and $v_m$ are the same then $P$ is a \textit{cycle}. A cycle $v_1\dots,v_m$ is called a \textit{loop} if $m=1$. Note that as $v_i,v_i$ is a loop of $G$ if and only if $e_{ii}\in E$ we may refer to the edge $e_{ii}$ as the loop. If $S\subseteq V$ where $V$ is the vertex set of a graph we will write $\bar{S}=V-S$.

The main idea behind an isospectral reduction of a graph $G=(V,E,\omega)$ is that we reduce $G$ to a smaller graph on some subset $S\subset V$. The sets $S$ for which this is possible are defined as follows.

\begin{definition}\label{def1}
Let $G=(V,E,\omega)$. A nonempty vertex set $S\subseteq V$ is a \textit{structural set} of $G$ if\\
(i) each cycle of $G$, that is not a loop, contains a vertex in $S$; and\\
(ii) $\omega(e_{ii})\neq \lambda$ for each $v_i\in \bar{S}$.
\end{definition}

Notice that part (i) of definition \ref{def1} states that a structural set $S$ of $G$ depends intrinsically on the structure of $G$. Part (ii), on the other hand, is the formal assumption that the loops of the vertices in $\bar{S}$, i.e. the complement of $S$, do not have weight equal to $\lambda\in\mathbb{W}[\lambda]$. That is, we mean such loops are not weighted by the rational function $\lambda/1\in\mathbb{W}[\lambda]$.

For $G\in\mathbb{G}$ we let $st(G)$ denote the set of all structural sets of the graph $G$.

\begin{definition}
Suppose $G=(V,E,\omega)$ with structural set $S=\{v_1,\dots,v_m\}$. Let $\mathcal{B}_{ij}(G;S)$ be the set of paths or cycles from $v_i$ to $v_j$ with no interior vertices in $S$. We call a path or cycle $\beta\in\mathcal{B}_{ij}(G;S)$ a \textit{branch} of $G$ with respect to $S$. We let
$$\mathcal{B}_S(G)=\bigcup_{1\leq i,j \leq m} \mathcal{B}_{ij}(G;S)$$
denote the set of all branches of $G$ with respect to $S$.
\end{definition}

If $\beta=v_1,\dots,v_m$ is a branch of $G$ with respect to $S$ and $m>2$ define
\begin{equation}\label{eq0.9}
\mathcal{P}_{\omega}(\beta)=\omega(e_{12})\prod_{i=2}^{m-1}\frac{\omega(e_{i,i+1})}{\lambda-\omega(e_{ii})}.
\end{equation}
For $m=1,2$ let $\mathcal{P}_{\omega}(\beta)=\omega(e_{1m})$. We call $\mathcal{P}_{\omega}(\beta)$ the \textit{branch product} of $\beta$. Note that assumption (ii) in definition \ref{def1} implies that the branch product of any $\beta\in\mathcal{B}_S(G)$ is always defined and is a rational function in $\mathbb{W}[\lambda]$.

In our procedure, which we term an \textit{isospectral graph reduction}, we replace the branches $\mathcal{B}_{ij}(G;S)$ of a graph with a single edge. The following definition specifies the weights of these edges.

\begin{definition}\label{IR}
Let $G=(V,E,\omega)$ with structural set $S=\{v_1\,\dots,v_m\}$. Define the edge weights
\begin{equation}\label{eq1.0}
\mu(e_{ij})=\begin{cases}
\displaystyle{\sum_{\beta\in\mathcal{B}_{ij}(G;S)}\mathcal{P}_\omega(\beta)} & \text{if} \ \ \ \mathcal{B}_{ij}(G;S)\neq\emptyset\\
\ \ \ \ \ 0 & \text{otherwise}
            \end{cases} \ \ \ \text{for} \ \ \ 1\leq i,j\leq m.
\end{equation}
The graph $\mathcal{R}_S(G)=(S,\mathcal{E},\mu)$ where $e_{ij}\in \mathcal{E}$ if $\mu(e_{ij})\neq 0$
is the \textit{isospectral reduction} of $G$ over $S$.
\end{definition}

Observe that $\mu(e_{ij})$ in definition \ref{IR} is the weight of the edge $e_{ij}$ in $\mathcal{R}_S(G)$. Moreover, as $W[\lambda]$ is closed under both addition and multiplication then the edge weights $\mu(e_{ij})$ of $\mathcal{R}_S(G)$ are also in the set $W[\lambda]$. Hence, the isospectral reduction $\mathcal{R}_S(G)$ is again a graph in $\mathbb{G}$.

\begin{example}\label{ex3}
Consider the graph $G=(V,E,\omega)$ given in figure \ref{fig2} (left) where each edge of $G$ is given unit weight. Note that the vertex set $S=\{v_1,v_3\}\subset V$ is a structural set of $G$ since\\
(i) the three nonloop cycles of $G$, namely $v_1,v_2,v_3,v_4,v_1$; $v_1,v_5,v_1$; and $v_3,v_6,v_3$ each contain a vertex in $S$; and\\
(ii) the loop weights of vertices in $\bar{S}=\{v_2,v_4,v_5,v_6\}$ are $\omega(e_{22})=1$, $\omega(e_{44})=1,$ $\omega(e_{55})=1$, and $\omega(e_{66})=1$ respectively. Hence, $\omega(e_{ii})=1\in\mathbb{W}[\lambda]$ is not equal to the rational function $\lambda/1\in\mathbb{W}[\lambda]$ for each $v_i\in\bar{S}$.

In contrast, the vertex set $T=\{v_1,v_5\}$ is not a structural set of $G$ as the (nonloop) cycle $v_3,v_6,v_3$ does not contain a vertex of $T$.

The branches in $\mathcal{B}_S(G)$ are respectively $\mathcal{B}_{11}(G;S)=\{v_1,v_5,v_1\}$, $\mathcal{B}_{13}(G;S)=\{v_1,v_2,v_3\}$, $\mathcal{B}_{31}(G;S)=\{v_3,v_4,v_1\}$, and $\mathcal{B}_{33}(G;S)=\{v_3,v_6,v_3\}$. Using equation (\ref{eq0.9}) the branch product of each branch is given by
$$\mathcal{P}_\omega(v_1,v_5,v_1)=\mathcal{P}_\omega(v_1,v_2,v_3)=\mathcal{P}_\omega(v_3,v_4,v_1)= \mathcal{P}_\omega(v_3,v_6,v_3)=\frac{1}{\lambda-1}.$$

Using equation (\ref{eq1.0}) each edge of $\mathcal{R}_S(G)=(S,\mathcal{E},\mu)$ has weight given by
$$\mu(e_{11})=\mu(e_{13})=\mu(e_{31})=\mu(e_{33})=\frac{1}{\lambda-1}.$$
As each edge weight is nonzero the edge set $\mathcal{E}$ of $\mathcal{R}_S(G)$ is $\mathcal{E}=\{e_{11}$, $e_{13}$, $e_{31}$, $e_{33}\}$. The graph $\mathcal{R}_S(G)$ is shown in figure \ref{fig2} (right).
\end{example}

\begin{figure}
  \begin{center}
    \begin{overpic}[scale=.5]{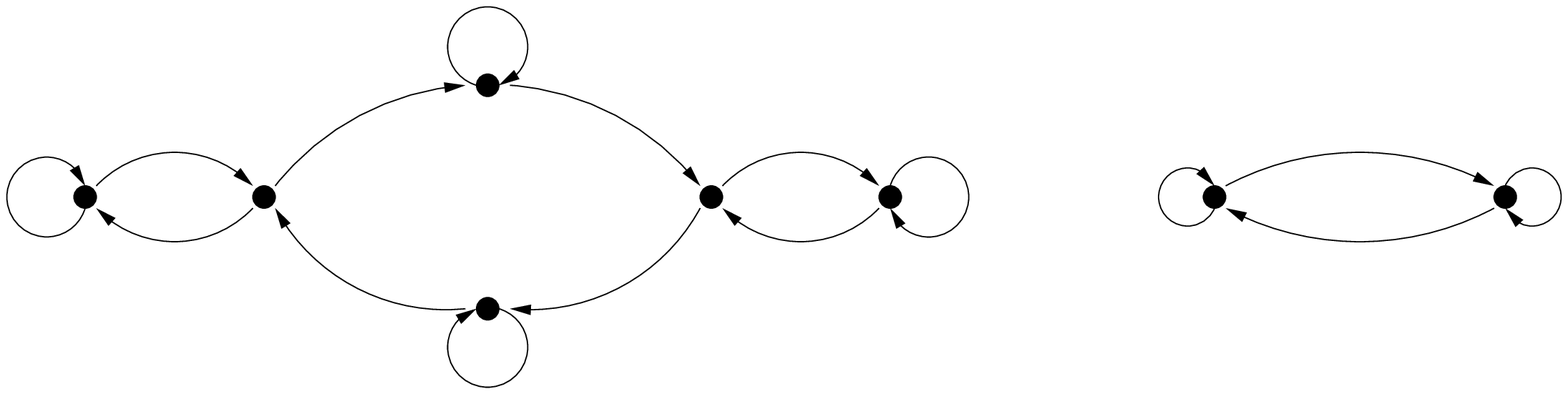}
    \put(29.5,-2.5){$G$}
    \put(15,15.5){$v_1$}
    \put(76.5,15.5){$v_1$}
    \put(29.5,17.25){$v_2$}
    \put(29.5,7.5){$v_4$}
    \put(4.5,15.5){$v_5$}
    \put(44,15.5){$v_3$}
    \put(94.5,15.5){$v_3$}
    \put(54.5,15.5){$v_6$}
    \put(84,18){{$\frac{1}{\lambda-1}$}}
    \put(84,6.5){$\frac{1}{\lambda-1}$}
    \put(68,12){$\frac{1}{\lambda-1}$}
    \put(100,12){$\frac{1}{\lambda-1}$}
    \put(82,-2.5){$\mathcal{R}_S(G)$}
    \end{overpic}
  \end{center}
  \caption{Reduction of $G$ over $S=\{v_1,v_3\}$ where each edge in $G$ has unit weight.}\label{fig2}
\end{figure}

Recall that if $S$ is a structural set of the graph $G\in\mathbb{G}$ then the isospectral reduction $\mathcal{R}_S(G)$ is also a graph in $\mathbb{G}$. Hence, both $G$ and $\mathcal{R}_S(G)$ have well-defined spectra. The relation between the spectrum $\sigma(G)$ and $\sigma(\mathcal{R}_S(G))$ is given in the following fundamental theorem.

\begin{theorem}\label{maintheorem}
Let $S$ be a structural set of the graph $G\in\mathbb{G}$. Then $$\sigma\big(\mathcal{R}_S(G)\big)=\big(\sigma(G)\cup\sigma^{-1}(G|_{\bar{S}})\big)- \big(\sigma(G|_{\bar{S}})\cup\sigma^{-1}(G)\big).$$
\end{theorem}

Since the cycles of $G|_{\bar{S}}$ are loops it follows that
\begin{equation}\label{eq0.2}
 \det\big(M(G|_{\bar{S}})-\lambda I\big)=\prod_{v_i\in \bar{S}}\big(\omega(e_{ii})-\lambda\big).
\end{equation}
In light of equation (\ref{eq0.2}), theorem \ref{maintheorem} has the following corollary.

\begin{corollary}\label{cor1}
Let $S$ be a structural set of the graph $G\in\mathbb{G}$. If $M(G)\in\mathbb{C}^{n\times n}$ then\\
(i) $\sigma(G|_{\bar{S}})=\{\omega(e_{ii}): v_i\in \bar{S}\}$;\\
(ii) $\sigma^{-1}(G|_{\bar{S}})=\emptyset$; and\\
(iii) $\sigma(\mathcal{R}_S(G))=\sigma(G)-\sigma(G|_{\bar{S}})$.
\end{corollary}

In many applications the graphs (matrices) that are used have real or positive weights (entries). If $G=(V,E,\omega)$ has complex valued weights and $S\in st(G)$ then corollary \ref{cor1} states that the spectrum of $\mathcal{R}_S(G)$ and $G$ differ at most by the spectrum of $G|_{\bar{S}}$. Moreover, the spectrum $\sigma(G|_{\bar{S}})$ are the weights of the loops $e_{ii}$ for $v_i\in\bar{S}$.

We note that theorem \ref{maintheorem} describes exactly which eigenvalues we may gain from an isospectral reduction and which me may lose. In this way an isospectral reduction of a graph preserves the spectral information of the original graph. This will be important in section 4 where we consider isospectral reductions that do not effect the nonzero eigenvalues of a graph.

\begin{example}
Let $G$ be the graph considered in example \ref{ex3}. As previously shown the vertex set $S=\{v_1,v_3\}$ is a structural set of $G$. Moreover, $M(G)\in\mathbb{C}^{6\times 6}$. Hence, corollary \ref{cor1} allows us to quickly compute the eigenvalues of the reduced graph $\mathcal{R}_S(G)$ once the eigenvalues of $G$ are known.

As one can calculate, the eigenvalues of the graph $G$ are $\sigma(G)=\{2,-1,1,1,1,0\}$. The restricted graph $G|_{\bar{S}}$ shown in figure \ref{fig3} has loop weights $\omega(e_{22})=1$, $\omega(e_{44})=1$, $\omega(e_{55})=1$, and $\omega(e_{66})=1$. Corollary \ref{cor1} therefore implies that $\sigma(G|_{\bar{S}})=\{1,1,1,1\}$ (and $\sigma^{-1}(G|_{\bar{S}})=\emptyset$). As $\sigma(\mathcal{R}_S(G))=\sigma(G)-\sigma(G|_{\bar{S}})$ then the spectrum of the reduced graph is $\sigma\big(\mathcal{R}_S(G)\big)=\{2,-1,0\}$.

Since the graph $\mathcal{R}_S(G)$ has two vertices the matrix $M(\mathcal{R}_S(G))\in\mathbb{W}[\lambda]^{2\times 2}$. However, notice that
$$\det\big(M(\mathcal{R}_S(G))-\lambda I\big)=\frac{\lambda^3-\lambda^2-2\lambda}{\lambda-1}$$
which is zero for $\lambda=2,-1,0$. This is an explicit demonstration of the fact that an $n\times n$ matrix in $\mathbb{W}[\lambda]^{n\times n}$ may have more than $n$ eigenvalues.

Therefore, the effect of reducing $G$ over $S$ is that we lose the eigenvalues $\{1,1,1\}$. However, even if $\sigma(G)$ is unknown we still know the following. The set of eigenvalues $\sigma(G|_{\bar{S}})=\{1,1,1,1\}$ is the most by which $\sigma(\mathcal{R}_S(G))$ and $\sigma(G)$ can differ.
\end{example}

\begin{figure}
  \begin{center}
    \begin{overpic}[scale=.45]{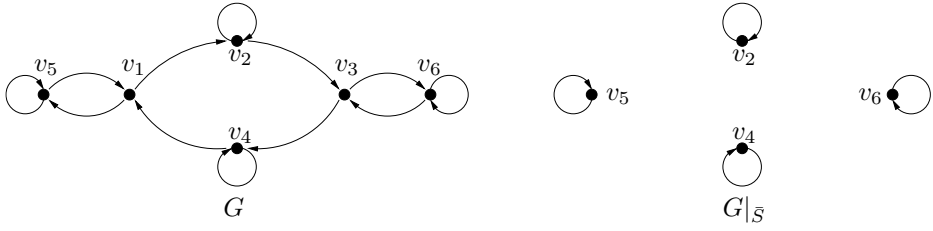}
    \put(22.5,-2.5){$G$}
    \put(12,12.5){$v_1$}
    \put(23,13.5){$v_2$}
    \put(23,5.5){$v_4$}
    \put(3,12.5){$v_5$}
    \put(34,12.5){$v_3$}
    \put(42.5,12.5){$v_6$}

    \put(62,9.5){$v_5$}
    \put(75,13.5){$v_2$}
    \put(75,5.5){$v_4$}
    \put(88,9.5){$v_6$}
    \put(74,-2.5){$G|_{\bar{S}}$}
    \end{overpic}
  \end{center}
  \caption{Restriction of the graph $G$ to $\bar{S}=\{v_2,v_4,v_5,v_6\}$ where each edge in $G$ and $G|_{\bar{S}}$ has unit weight.}\label{fig3}
\end{figure}

We note that both $\sigma(G|_{\bar{S}})$ and $\sigma^{-1}(G|_{\bar{S}})$ are easily calculated via equation (\ref{eq0.2}). Therefore, theorem \ref{maintheorem} offers a quick way of computing the eigenvalues of a reduced graph if the spectrum of the original unreduced graph is known.

\subsection{Sequential Reductions}
In section 2.1 we observed that any reduction $\mathcal{R}_S(G)$ of a graph $G\in\mathbb{G}$ is again a graph in $\mathbb{G}$. It is therefore possible to consider sequential reductions of a graph $G\in\mathbb{G}$. However, this requires that we first extend our notation to sequences of isospectral reductions.

For $G=(V,E,\omega)$ suppose $S_m\subseteq S_{m-1}\subseteq\dots\subseteq S_1\subseteq V$ such that $S_1\in st(G)$, $\mathcal{R}_1(G)=\mathcal{R}_{S_1}(G)$ and $$S_{i+1}\in st(\mathcal{R}_i(G)) \ \text{where} \ \mathcal{R}_{S_{i+1}}(\mathcal{R}_i(G))=\mathcal{R}_{i+1}(G), \  1\leq i\leq m-1.$$ If this is the case we say $S_1,\dots,S_m$ \textit{induces a sequence of reductions} on $G$ with \textit{final vertex set} $S_m$. By way of notation we write $\mathcal{R}_m(G)=\mathcal{R}(G;S_1,\dots,S_m)$ where $\mathcal{R}(G;S_1,\dots,S_m)$ denotes the graph $G$ reduced over the vertex set $S_1$ then $S_2$ and so on until $G$ is reduced over the final vertex set $S_m$.

If $S\notin st(G)$ it is natural to ask whether there exists a sequence of vertex sets satisfying $S\subseteq S_{m-1}\subseteq\dots\subseteq S_{1}\subseteq V$ that induce a sequence of reductions on $G$ and if so is this the only such sequence? To answer these questions we note the following.

If the weight $\omega(e_{ii})\neq\lambda$ for some $v_i\in V$ then the vertex set $S=V-\{v_i\}$ is a structural set of $G$. This follows from the fact that $\bar{S}=\{v_i\}$. Hence, the graph $G|_{\bar{S}}$ is the graph restricted to the single vertex $v_i$. In particular, this implies that any cycle of $G|_{\bar{S}}$ is a loop.

Therefore, any graph $G\in\mathbb{G}$ can be reduced over the structural set $S=V-\{v_i\}$ if it is known that $\omega(e_{ii})\neq\lambda$. Another way to state this is that it is possible to remove the vertex $v_i$ from $G$ via an isospectral reduction if $\omega(e_{ii})\neq\lambda$ without knowing anything about the graph structure of $G$. This has the following important implication.

Suppose it is known that no loop of $G$ or any loop of any sequential reduction of $G$ has weight $\lambda$. If this is the case then it is possible to remove any sequence of single vertices from $G$ via a sequence of isospectral reductions. Therefore, $G$ can be sequentially reduced to a graph on any subset of its vertex set. This idea is the motivation behind the following.

For any polynomial $p\in\mathbb{C}[\lambda]$ let $deg(p)$ denote the degree of $p$. If $w=p/q\in\mathbb{W}[\lambda]$ where $p,q\in\mathbb{C}[\lambda]$ let $$\pi(w)=deg(p)-deg(q).$$

Let $\mathbb{W}_\pi[\lambda]$ be the subset of $\mathbb{W}[\lambda]$ given by
$$\mathbb{W}_\pi[\lambda]=\big\{w\in\mathbb{W}[\lambda]:\pi(w)\leq 0\big\}.$$
That is, $\mathbb{W}_\pi[\lambda]$ is the set of rational functions in which the degree of the numerator is less than or equal to the degree of the denominator. Let $\mathbb{G}_\pi$ be the graphs in $\mathbb{G}$ with edge weights in the set $\mathbb{W}_\pi[\lambda]$.

\begin{lemma}\label{lemma1}
If $G\in\mathbb{G}_\pi$ and $S\in st(G)$ then $\mathcal{R}_S(G)\in\mathbb{G}_\pi$. In particular, no loop of $G$ and no loop of any reduction of $G$ can have weight $\lambda$.
\end{lemma}

By the reasoning above, if $G\in\mathbb{G}_\pi$ then $G$ can be (sequentially) reduced to a graph on any subset of its vertex set. This result is stated in the following theorem.

\begin{theorem}\label{theorem-1}\textbf{(Existence of Isospectral Reductions Over any Vertex Set)} Let $G=(V,E,\omega)$ be graph in $\mathbb{G}_\pi$ and suppose $\mathcal{V}$ is a nonempty subset of $V$. Then there exist sets $\mathcal{V}\subseteq S_{m-1}\subseteq\dots\subseteq S_1\subseteq V$ such that $S_1,\dots, S_{m-1},\mathcal{V}$ induces a sequence of reductions on $G$.
\end{theorem}

For $G\in\mathbb{G}_\pi$ it is therefore possible to reduce a graph $G\in\mathbb{G}_{\pi}$ to a graph on any (nonempty) subset of vertex set via some sequence of isospectral reductions. Moreover, such reductions have the following uniqueness property.

\begin{theorem}\label{theorem-3}\textbf{(Uniqueness of Isospectral Reductions Over any Vertex Set)} Let $G=(V,E,\omega)$ be graph in $\mathbb{G}_\pi$ and suppose $\mathcal{V}$ is a nonempty subset of $V$. If $S_1,\dots, S_{m-1},\mathcal{V}$ and $T_1,\dots, T_{n-1},\mathcal{V}$ both induce a sequence of reductions on $G$ then $\mathcal{R}(G;S_1,\dots, S_{m-1},\mathcal{V})=\mathcal{R}(G;T_1,\dots, T_{n-1},\mathcal{V})$.
\end{theorem}

The results of theorem \ref{theorem-1} and theorem \ref{theorem-3} allow for the following definition.

\begin{definition}\label{note}
Let $G=(V,E,\omega)$ be graph in $\mathbb{G}_\pi$. If $\mathcal{V}\subseteq V$ is nonempty define
$$\mathcal{R}_{\mathcal{V}}[G]=\mathcal{R}(G;S_1,\dots, S_{m-1},\mathcal{V})$$
where $S_1,\dots, S_{m-1},\mathcal{V}$ is any sequence that induces a sequence of reductions on $G$ with final vertex set $\mathcal{V}$.
\end{definition}

The graph $\mathcal{R}_{\mathcal{V}}[G]$ is well defined as a result of theorems \ref{theorem-1} and \ref{theorem-3}. The notation $\mathcal{R}_\mathcal{V}[G]$ given in definition \ref{note} is intended to emphasize the fact that $\mathcal{V}$ need not be a structural set of $G$.

\begin{remark}
Note that $\pi(c)=0$ for any $c\in\mathbb{C}$. Hence, if $M(G)\in\mathbb{C}^{n\times n}$ then $G\in\mathbb{G}_\pi$. Therefore, any graph with complex weights can be uniquely reduced to a graph on any nonempty subset of its vertex set. This is of particular importance for the estimation of spectra of matrices with complex entries in \cite{BW09}.
\end{remark}

\begin{figure}
  \begin{center}
    \begin{overpic}[scale=.5]{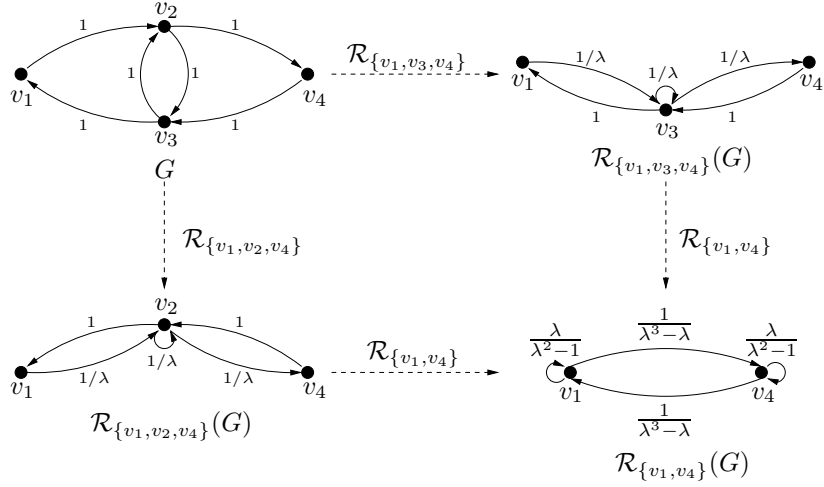}
    \put(17.5,35.5){$G$}
    \put(8,54){\tiny$1$}
    \put(22,48){\tiny$1$}
    \put(27,54){\tiny$1$}
    \put(8,41.5){\tiny$1$}
    \put(14,48){\tiny$1$}
    \put(27,41.5){\tiny$1$}
    \put(-.5,45){$v_1$}
    \put(17.5,56){$v_2$}
    \put(17.5,40){$v_3$}
    \put(36,45){$v_4$}

    \put(62,47){$v_1$}
    \put(80,41){$v_3$}
    \put(98,47){$v_4$}
    \put(72,42.5){\tiny$1$}
    \put(89,42.5){\tiny$1$}
    \put(70,50){\tiny$1/\lambda$}
    \put(88,50){\tiny$1/\lambda$}
    \put(79,48){\tiny$1/\lambda$}
    \put(72,37){$\mathcal{R}_{\{v_1,v_3,v_4\}}(G)$}

    \put(-.5,8.5){$v_1$}
    \put(17.5,19){$v_2$}
    \put(36,8.5){$v_4$}
    \put(8,10){\tiny$1/\lambda$}
    \put(26,10){\tiny$1/\lambda$}
    \put(16.5,12){\tiny$1/\lambda$}
    \put(9,17){\tiny$1$}
    \put(27.5,17){\tiny$1$}
    \put(9.5,4){$\mathcal{R}_{\{v_1,v_2,v_4\}}(G)$}

    \put(68,8){$v_1$}
    \put(92,8){$v_4$}
    \put(64,14.5){$\frac{\lambda}{\lambda^2-1}$}
    \put(91,14.5){$\frac{\lambda}{\lambda^2-1}$}
    \put(77.5,4.5){$\frac{1}{\lambda^3-\lambda}$}
    \put(77.5,16.5){$\frac{1}{\lambda^3-\lambda}$}
    \put(75,-1){$\mathcal{R}_{\{v_1,v_4\}}(G)$}

    \put(41.5,50){$\mathcal{R}_{\{v_1,v_3,v_4\}}$}
    \put(21,27){$\mathcal{R}_{\{v_1,v_2,v_4\}}$}
    \put(44,13){$\mathcal{R}_{\{v_1,v_4\}}$}
    \put(83,27){$\mathcal{R}_{\{v_1,v_4\}}$}

    \end{overpic}
  \end{center}
  \caption{Distinct sequences of isospectral reductions with the same final vertex set and outcome.}\label{fig6}
\end{figure}

\begin{example}
Let $G=(V,E,\omega)$ be the graph shown in figure \ref{fig6}. Our goal is to reduce $G$ over the vertex set $\{v_1,v_4\}\subset V$. Note that as $G\in\mathbb{G}_\pi$ theorem \ref{theorem-1} guarantees that there is at least one sequence of reductions that reduces $G$ to the graph $\mathcal{R}_{\{v_1,v_4\}}[G]$.

In fact there are exactly two. This follows from the fact that $\{v_1,v_4\}\notin st(G)$. Hence, $G$ cannot be reduced over $\{v_1,v_4\}$ with a single reduction. However, any (nontrivial) reduction of $G$ removes at least one vertex from $G$.

Therefore, the two possible ways of reducing $G$ over the vertex set $\{v_1,v_4\}$ are
\begin{align}\label{Al1}
\mathcal{R}_{\{v_1,v_4\}}[G]&=\mathcal{R}(G;\{v_1,v_2,v_4\},\{v_1,v_4\}); \ \ \text{and}\\
\mathcal{R}_{\{v_1,v_4\}}[G]&=\mathcal{R}(G;\{v_1,v_3,v_4\},\{v_1,v_4\}).\label{Al2}
\end{align}

Both of the reductions given in (\ref{Al1}) and (\ref{Al2}) are shown in figure \ref{fig6}. The dashed arrows labeled $\mathcal{R}_T$ in this figure represent the reduction of a graph over some structural set $T$. This notation is meant to emphasize that this diagram commutes. That is,
$$\mathcal{R}_{\{v_1,v_4\}}\big(\mathcal{R}_{\{v_1,v_2,v_4\}}(G)\big)=\mathcal{R}_{\{v_1,v_4\}}\big(\mathcal{R}_{\{v_1,v_3,v_4\}}(G)\big)$$
as guaranteed by theorem \ref{theorem-3}.
\end{example}

\subsection{Equivalence Relations}

Theorem \ref{theorem-1} and theorem \ref{theorem-3} assert that a graph $G\in\mathbb{G}_\pi$ has a unique reduction to any (nonempty) subset of its vertex set via some sequence of isospectral reductions. In this section this property will allow us to define various equivalence relations on the graphs in $\mathbb{G}_{\pi}-\emptyset$.

Two weighted digraphs $G_1=(V_1,E_1,\omega_1)$, and $G_2=(V_2,E_2,\omega_2)$ are \textit{isomorphic} if there is a bijection $b:V_1\rightarrow V_2$ such that there is an edge $e_{ij}$ in $G_1$ from $v_i$ to $v_j$ if and only if there is an edge $\tilde{e}_{ij}$ between $b(v_i)$ and $b(v_j)$ in $G_2$ with $\omega_2(\tilde{e}_{ij})=\omega_1(e_{ij})$. If the map $b$ exists it is called an \textit{isomorphism} and we write $G_1\simeq G_2$.

An isomorphism is essentially a relabeling of the vertices of a graph. Therefore, if two graphs are isomorphic then their spectra are identical.

\begin{figure}
  \begin{center}
    \begin{overpic}[scale=.33]{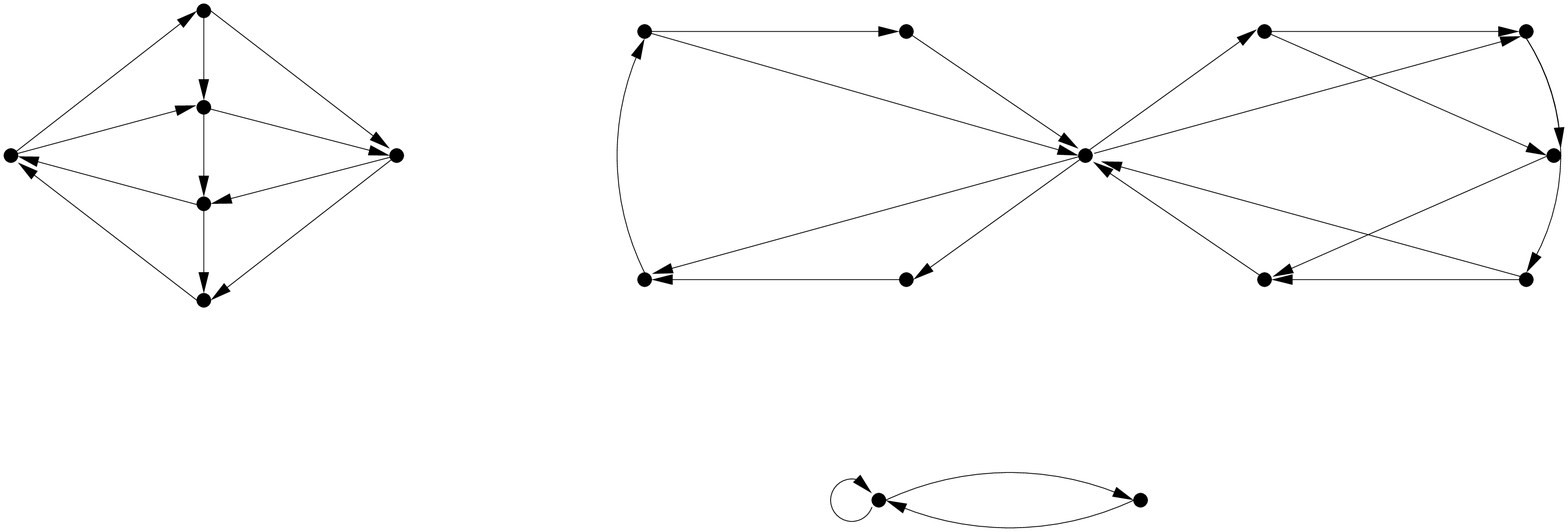}
    \put(-1,27){$v_1$}
    \put(24,27){$v_2$}
    \put(12,17){$G$}

    \put(67,17){$H$}
    \put(68,27){$v_1$}
    \put(101,30){$v_2$}

    \put(35,7.5){$\frac{1}{\lambda^2}+\frac{2}{\lambda^3}+\frac{1}{\lambda^4}$}
    \put(60,11.75){$\frac{2}{\lambda}+\frac{1}{\lambda^3}$}
    \put(60,3.5){$\frac{2}{\lambda}+\frac{1}{\lambda^3}$}
    \put(40,-1.5){$\mathcal{R}_{\tau(G)}[G]\simeq\mathcal{R}_{\tau(H)}[H]$}
    \put(55,5){$v_1$}
    \put(71,5){$v_2$}
    \end{overpic}
  \end{center}
  \caption{$G$ and $H$ are equivalent under the relation induced by the rule $\tau$ given in example \ref{ex0}.}\label{fig7}
\end{figure}

\begin{theorem}\textbf{(Spectral Equivalence)}\label{theorem0.1}
Suppose for any graph $G=(V,E,\omega)$ in $\mathbb{G}_\pi-\emptyset$ that $\tau$ is a rule that selects a unique nonempty subset $\tau(G)\subseteq V$. Then $\tau$ induces an equivalence relation $\sim$ on the set $\mathbb{G}_\pi-\emptyset$ where $G\sim H$ if the graph $\mathcal{R}_{\tau(G)}[G]\simeq\mathcal{R}_{\tau(H)}[H]$.
\end{theorem}

Two graphs may look very different but be spectrally equivalent, in which case the corresponding networks have similar dynamics. Choosing an appropriate rule $\tau$ can help discover this similarity.

\begin{example}\label{ex0}
Let $G=(V,E,\omega)$. For $v_i\in V$ let $c(v_i)$ be the number of cycles in $G$ that contain $v_i$. If $c_{max}(G)=\max_{v_i\in V}c(v_i)$ let $$\tau(G):=\{v_i\in V: c(v_i)\geq c_{max}(G)/2\}.$$

Observe that for each graph $G\in\mathbb{G}_\pi-\emptyset$ the set $\tau(G)$ both exists and is unique. Thus the relation of having an isomorphic reduction with respect to this rule induces an equivalence relation on $\mathbb{G}_\pi-\emptyset$.

In figure \ref{fig7} the graphs $G$ and $H$ have the vertex set $\tau(G)=\{v_1,v_2\}=\tau(H)$. As shown in the figure, the graph $\mathcal{R}_{\tau(G)}[G]\simeq\mathcal{R}_{\tau(H)}[H]$. Hence, $G\sim H$ under the relation $\sim$ induced by the rule $\tau$.
\end{example}

Importantly, choosing a rule that selects a unique vertex set of each graph allows one to study the graphs in $\mathbb{G}_\pi-\emptyset$ modulo some particular graph feature. (In example \ref{ex0} this graph feature or vertex set are the vertices that are part of less than $c_{max}(G)/2$ cycles of $G$.) This allows experimentalists the opportunity to compare the \emph{reduced topology} of various networks which have been reduced over any subset of network elements deemed important in the particular research field or by a particular researcher.

\section{Weight Preserving Isospectral Transformations}

The isospectral graph reductions introduced in section 2 modify not only the graph structure but also the weight set of a graph. That is, if $\mathcal{R}_S(G)=(S,\mathcal{E},\mu)$ is any reduction of $G=(V,E,\omega)$ then typically $\omega(E)\neq \mu(\mathcal{E})$, i.e.
$$\{\omega(e_{ij}):e_{ij}\in E\}\neq\{\mu(e_{ij}):e_{ij}\in \mathcal{E}\}.$$
This may lead one to assume that this procedure simply shifts the complexity of the graph's structure to its set of edge weights. However, this is not the case.

In this section we introduce graph transformations that modify the structure of a graph but preserve the weights of the graph's edges. As before, the procedure preserves the spectrum of the graph up to a known set. Such transformations are of particular importance in section 4.2 where dynamical network expansions are discussed.

The idea behind an isospectral graph transformation that preserves a graph's edge weights is simple enough. If two graphs $G,H\in\mathbb{G}$ have the same branch structure (including weights) then they should have similar spectra.

To make this precise suppose $G=(V,E,\omega)$ and $S\in st(G)$. If the branch $\beta=v_{1},\dots,v_{m}\in\mathcal{B}_S(G)$ let $\Omega_G(\beta)$ be the ordered sequence $$\Omega_G(\beta)=\omega(e_{12}),\dots,\omega(e_{i-1,i}),\omega(e_{ii}),\omega(e_{i,i_+1}),\dots,\omega(e_{m-1,m}).$$ for $m>1$ and $\omega(e_{ii})$ if $m=1$.

Let $G,H\in\mathbb{G}$. Suppose $S=\{v_1,\dots,v_m\}$ is a structural set of both $G$ and $H$. The branch set $\mathcal{B}_{ij}(G;S)$ is \textit{isomorphic} to $\mathcal{B}_{ij}(H;S)$ if there is a bijection
$$b:\mathcal{B}_{ij}(G;S)\rightarrow\mathcal{B}_{ij}(H;S)$$
such that $\Omega_G(\beta)=\Omega_H(b(\beta))$ for each $\beta\in\mathcal{B}_{ij}(G;S)$. If such a map exists we write $\mathcal{B}_{ij}(G;S)\simeq\mathcal{B}_{ij}(H;S)$. If
$$\mathcal{B}_{ij}(G;S)\simeq\mathcal{B}_{ij}(H;S) \ \ \text{for each} \ \ 1\leq i,j\leq m$$
we say $\mathcal{B}_S(G)$ is \textit{isomorphic} to $\mathcal{B}_S(H)$ and write $\mathcal{B}_S(G)\simeq\mathcal{B}_S(H)$.

\begin{definition}
If $S$ is a structural set of both $G,H\in\mathbb{G}$ and $\mathcal{B}_S(G)\simeq\mathcal{B}_S(H)$ we say $G$ is a \emph{weight preserving isospectral transformation} of $H$ over $S$.
\end{definition}

\begin{example}\label{ex9}
Suppose $H$ and $G$ are the graphs in figure \ref{fig9}. Note that the vertex set $S=\{v_1,v_3\}$ is a structural set of both $H$ and $G$. Moreover,
\begin{align*}
\mathcal{B}_{11}(H;S)=\{v_1,w_2,v_1\},& \hspace{0.15in} \mathcal{B}_{11}(G;S)=\{v_1,v_5,v_1\};\\
\mathcal{B}_{13}(H;S)=\{v_1,w_2,v_3\},& \hspace{0.15in} \mathcal{B}_{13}(G;S)=\{v_1,v_2,v_3\};\\
\mathcal{B}_{31}(H;S)=\{v_3,w_4,v_1\},& \hspace{0.15in} \mathcal{B}_{31}(G;S)=\{v_3,v_4,v_1\};\\
\mathcal{B}_{33}(H;S)=\{v_3,w_2,v_3\},& \hspace{0.15in} \mathcal{B}_{33}(G;S)=\{v_3,v_6,v_3\}.
\end{align*}

Note that the branch $\beta=\{v_1,v_5,v_1\}$ in $\mathcal{B}_{11}(G;S)$ has the weight sequence given by $\Omega_G(\beta)=1,1,1$. Similarly, the branch $\gamma=\{v_1,w_2,v_1\}$ in $\mathcal{B}_{11}(H;S)$ has the weight sequence $\Omega_H(\gamma)=1,1,1$. Hence, $\mathcal{B}_{11}(H;S)\simeq\mathcal{B}_{11}(G;S)$. Continuing in this manner, one can check that each $\mathcal{B}_{ij}(H;S)\simeq\mathcal{B}_{ij}(G;S)$ for $i,j\in\{1,3\}$. Therefore, $\mathcal{B}_S(H)\simeq\mathcal{B}_S(G)$ or $G$ is a weight preserving isospectral transformation of $H$ over $S$.
\end{example}

\begin{figure}
  \begin{center}
    \begin{overpic}[scale=.5]{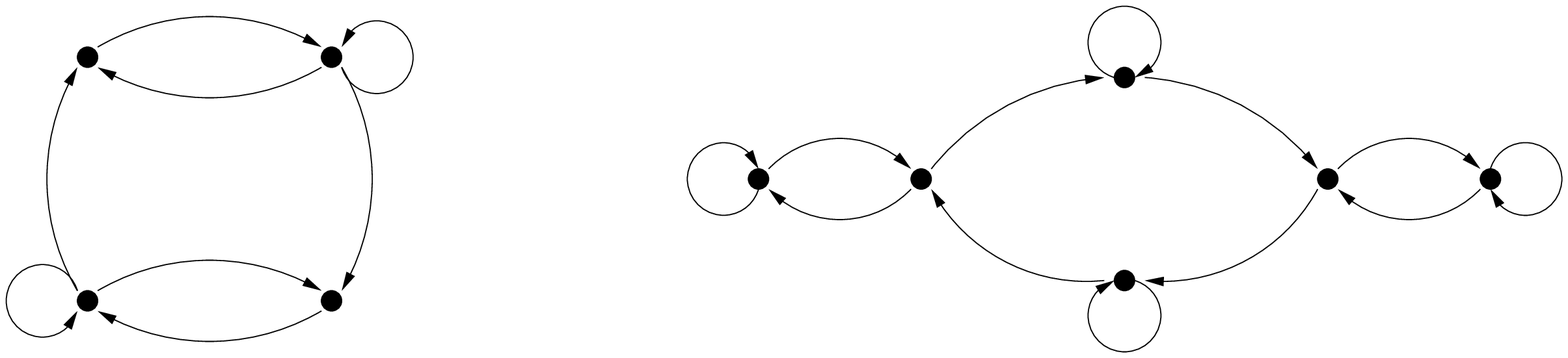}
    \put(4,21.5){$v_1$}
    \put(18.5,21.5){$w_2$}
    \put(20,1){$v_3$}
    \put(4.5,1){$w_4$}
    \put(11.5,-2){$H$}

    \put(47.5,14){$v_5$}
    \put(57.5,14){$v_1$}
    \put(70.5,16){$v_2$}
    \put(70.5,6.5){$v_4$}
    \put(83.5,14){$v_3$}
    \put(93.25,14){$v_6$}
    \put(70.5,-2){$G$}

    \end{overpic}
  \end{center}
  \caption{The graph $G$ is a isospectral expansion of the graph $H$ with respect to the structural set $S=\{v_1,v_3\}$.}\label{fig9}
\end{figure}

\begin{remark}
We note that if $\mathcal{B}_S(H)\simeq\mathcal{B}_S(G)$ the graphs $G$ and $H$ need not be isomorphic (see example \ref{ex9}).
\end{remark}

As $\mathcal{B}_S(G)\simeq\mathcal{B}_S(H)$ implies that $\mathcal{R}_S(G)=\mathcal{R}_S(H)$ we have the following corollary to theorem \ref{maintheorem}.

\begin{corollary}
If the graph $G$ is a weight preserving isospectral transformation of $H$ over $S$ then
$$\big(\sigma(G)\cup\sigma^{-1}(G|_{\bar{S}})\big)-\big(\sigma(G|_{\bar{S}}\cup\sigma^{-1}(G)\big)=
\big(\sigma(H)\cup\sigma^{-1}(H|_{\bar{S}})\big)-\big(\sigma(H|_{\bar{S}}\cup\sigma^{-1}(H)\big).$$
\end{corollary}

For $G\in\mathbb{G}$ and $S\in st(G)$ suppose $\alpha=v_1,\dots,v_m$ and $\beta=u_1,\dots,u_n$ are branches in $\mathcal{B}_S(G)$. These branches are said to be \textit{independent} if $$\{v_2,\dots,v_{m-1}\}\cap\{u_2,\dots,u_{n-1}\}= \emptyset.$$
That is, $\alpha$ and $\beta$ are independent if they share no interior vertices.

\begin{definition}\label{def10}
Let $G,H\in\mathbb{G}$ and $S\in st(G),st(H)$. Suppose\\
(i) $\mathcal{B}_S(G)\simeq\mathcal{B}_T(H)$;\\
(ii) the branches of $\mathcal{B}_S(H)$ are independent;\\
(iii) each vertex of $G$ and $H$ belongs to a branch of $\mathcal{B}_S(G)$ and $\mathcal{B}_S(H)$ respectively.
Then we call $H$ an \textit{isospectral expansion} of $G$ with respect to $S$.
\end{definition}

Isospectral expansions are particular types of weight preserving isospectral transformations and moreover are unique up to a labeling of vertices. Therefore, any two expansions of $G$ with respect to $S$ are isomorphic. By slight abuse of terminology we let $\mathcal{X}_S(G)$ be any representative from the set of isospectral expansions and call $\mathcal{X}_S(G)$ the \textit{isospectral expansion} of $G$ with respect to $S$.

\begin{theorem}\label{theorem-2}
Let $G=(V,E,\omega)$ with structural set $S$. Then the graph $G$ and its isospectral expansion $\mathcal{X}_S(G)$ have the same set of edge weights. Moreover,  $$\det\big(M(\mathcal{X}_S(G))-\lambda I \big)=\det\big(M(G)-\lambda I\big)\prod_{v_i\in V-S}\big(\omega(e_{ii})-\lambda\big)^{n_i-1}$$ where $n_i$ is the number of branches in $\mathcal{B}_S(G)$ containing $v_i$.
\end{theorem}

\begin{example}\label{ex10}
Consider the graph $G=(V,E,\omega)$ and $H=(\mathcal{V},\mathcal{E},\mu)$ in figure \ref{fig9}. As demonstrated in example \ref{ex9}, if $S=\{v_1,v_3\}$ then the branch set $\mathcal{B}_S(G)\simeq\mathcal{B}_S(H)$. Moreover, it can be seen from figure \ref{fig9} that the four branches of $\mathcal{B}_S(G)$ share no interior vertices. That is, the branches of $\mathcal{B}_S(G)$ are pairwise independent. Lastly, each vertex of $G$ belongs to at least one branch of $\mathcal{B}_S(G)$.

Therefore, the graph $G$ is an isospectral expansion of $H$ with respect to $S$ or $G=\mathcal{X}_S(H)$. Importantly, note that the edge weights of $H$ and its expansion $G$ are identical, i.e. the sets $\omega(E)$ and $\mu(\mathcal{E})$ are both $\{1\}$.

Moreover, observe that the vertex $w_2$ of $H$ is an interior vertex of the branches $v_1,w_2,v_1;$ $v_1,w_2,v_3\in\mathcal{B}_S(H)$. Similarly, the vertex $w_4$ of $H$ is an interior vertex of the branches $v_3,w_4,v_3;$ $v_3,w_4,v_1\in\mathcal{B}_S(H)$. As $\omega(e_{22})=1$ and $\omega(e_{44})=1$ theorem \ref{theorem-2} implies that $\sigma\big(\mathcal{X}_S(H)\big)=\sigma(H)\cup\{1,1\}.$

We note that $G=\mathcal{X}_S(H)$ in figure \ref{fig9} is the graph $G$ in figure \ref{fig2}. Hence, $\sigma(G)=\{2,-1,1,1,1,0\}$ implying $\sigma(H)=\{2,-1,1,0\}$.
\end{example}

The principle idea behind an isospectral expansion is the following. If $G\in\mathbb{G}$ and $S\in st(G)$ then the set of branches $\mathcal{B}_S(G)$ is uniquely defined. However, there are typically many other graphs $H$ with the same branch structure as $G$, i.e. $S\in st (H)$ such that $\mathcal{B}_S(H)\simeq\mathcal{B}_S(G)$.

An isospectral expansion of $G$ over $S$ is then a graph $H=\mathcal{X}_S(G)$ with identical branch structure but with the following restriction: The branches of $\mathcal{B}_S(H)$ are pairwise independent and every vertex of $H$ belongs to a branch in $\mathcal{B}_S(H)$. That is, any vertex of $\bar{S}$ in $H$ is part of exactly one branch in $\mathcal{B}_S(H)$.

Hence, given a graph $G$ and structural set $S$ we can algorithmically construct the expansion $\mathcal{X}_S(G)$ as follows. Start with the vertices $S$. If $\beta\in\mathcal{B}_{ij}(G;S)$ then both $v_i,v_j\in S$. Construct a path (or cycle) from $v_i$ to $v_j$ with weight sequence $\Omega(\beta)$ with \textit{new} interior vertices. By new we mean vertices that do not already appear on the graph we are constructing. Repeat this for each $\beta\in\mathcal{B}_{ij}(G;S)$. The resulting graph is the isospectral expansion $\mathcal{X}_S(G)$.

Importantly, an isospectral expansion is only one example of an isospectral graph transformation that preserves the weight set of a graph. Many other weight preserving isospectral transformations are possible. Moreover, other isospectral graph transformations that modify the weight set of a graph but restrict these weights to a particular subsets of $\mathbb{W}[\lambda]$ are also possible (see \cite{BW10}).

\section{Networks as Dynamical Systems}
As mentioned in the introduction, the dynamics of a network can be analyzed in terms of three key features; (i) the internal (local) dynamics of the network elements, (ii) the interactions between the network elements, and (iii) the topology or structure of the graph of interactions of the network.

To study the dynamics of networks our first task is to establish a mathematical framework for the investigation of networks as dynamical systems. This is done following the approach given in \cite{Afriamovich07}.

Let $i\in \mathcal{I}=\{1,\dots,n\}$ and $T_i: X_i\rightarrow X_i$ be maps on the complete metric space $(X_i,d)$ where
\begin{equation}\label{eq0.1}
L_i=\sup_{x_i\neq y_i\in X_i}\frac{d(T_i(x_i),T_i(y_i))}{d(x_i,y_i)}<\infty.
\end{equation}
Let $(T,X)$ denote the direct product of the local systems $(T_i,X_i)$ over $\mathcal{I}$ on the complete metric space $(X,d_{max})$ where for $\textbf{x},\textbf{y}\in X$
$$d_{max}(\textbf{x},\textbf{y})=\max_{i\in\mathcal{I}}\{d(x_i,y_i)\}.$$

\begin{definition}\label{interaction}
A map $F:X \rightarrow X$ is called an \textit{interaction} if for every $j\in \mathcal{I}$ there exists a nonempty collection of indices $\mathcal{I}_j\subseteq\mathcal{I}$ and a continuous function
$$F_j:\bigoplus_{i\in \mathcal{I}_j} X_i\rightarrow X_j,$$
that satisfies the following Lipschitz condition for constants $\Lambda_{ij}\geq 0:$
\begin{equation}\label{eq2.3}
d\big(F_j(\textbf{x}|_{\mathcal{I}_j}),F_j(\textbf{y}|_{\mathcal{I}_j})\big)\leq \sum_{i\in \mathcal{I}_j} \Lambda_{ij} d(x_i,y_i)
\end{equation}
for all $\textbf{x},\textbf{y}\in X$ where $\textbf{x}|_{\mathcal{I}_j}$ is the restriction of $\textbf{x}\in X$ to $\bigoplus_{i\in \mathcal{I}_j} X_i$.
Then the (\textit{interaction}) map $F$ is defined as follows:
$$F(\textbf{x})_j=F_j(\textbf{x}|_{\mathcal{I}_j}), \ \ j\in \mathcal{I}, \ \ i\in\mathcal{I}_j.$$
\end{definition}

\begin{definition}\label{netdef}
The superposition $\mathcal{F}=F\circ T$ generates the dynamical system $(\mathcal{F},X)$ which is a \textit{dynamical network}.
\end{definition}

The constants $\Lambda_{ij}$ in definition \ref{interaction} form the \textit{Lipschitz matrix} $\Lambda\in\mathbb{R}^{n\times n}$ where the entry $\Lambda_{ij}=0$ if $i\notin\mathcal{I}_j$.

\begin{remark}
Suppose the interaction $F:X\rightarrow X$ is continuously differentiable and each $X_i\subseteq\mathbb{R}$. If $DF$ is the matrix of first partial derivatives of $F$ then the constants
$$\Lambda_{ij}=\max_{\textbf{x}\in X}|(DF)_{ji}(\textbf{x})|$$
satisfy condition (\ref{eq2.3}) for the interaction $F$.
\end{remark}

\begin{definition}
Let $F:X\rightarrow X$ be an interaction. The graph $\Gamma_{F}=(V,E,\omega)$ with $V=\{v_1,\dots,v_n\}$, $E=\{e_{ij}: i\in \mathcal{I}_j, \ j\in\mathcal{I}\}$, and $\omega(e_{ij})=1$ for $e_{ij}\in E$ is called the \textit{graph of interactions} of $F$.
\end{definition}

We note that each vertex $v_i\in V$ of the graph $\Gamma_F=(V,E,\omega)$ corresponds to the $i$th component (coordinate) of the dynamical network $(\mathcal{F},X)$. Moreover, there is an edge $e_{ij}\in E$ if and only if the $j$th coordinate of the interaction $F(\textbf{x})$ depends on the $i$th coordinate of $\textbf{x}$.

\begin{example}\label{ex1}
Let $F:[0,1]^4\rightarrow[0,1]^4$ be the parameterized interaction given by
\begin{equation}\label{label1}
F(\mathbf{x};\alpha)=\left[
\begin{array}{c}
1-\alpha x_1x_4\\
1-\alpha x_1x_3\\
1-\alpha x_2x_3\\
1-\alpha x_1x_3
\end{array}
\right]  \ \ \text{for} \ \ \alpha\in[0,1].
\end{equation}
Using the Lipschitz constants $\Lambda_{ij}=\max_{\textbf{x}\in X}|(DF)_{ji}(\textbf{x})|$ the interaction $F$ has the Lipschitz matrix
\begin{equation}\label{label2}
\Lambda=\left[
\begin{array}{cccc}
\alpha&\alpha&0&\alpha\\
0&0&\alpha&0\\
0&\alpha&\alpha&\alpha\\
\alpha&0&0&0
\end{array}
\right].
\end{equation}
The graph of interactions $\Gamma_F$ of $F$ is the graph shown in figure \ref{fig2001} (left).
\end{example}

\subsection{Stability of Dynamical Networks}
 In this section we give sufficient conditions under which a dynamical network $(\mathcal{F},X)$ has simple dynamics. By simple we mean that the system $(\mathcal{F},X)$ has a globally attracting fixed point.

\begin{definition}
The dynamical network $(\mathcal{F},X)$ has a \textit{globally attracting fixed point} $\tilde{\textbf{x}}\in X$ if for any $\textbf{x}\in X$, $$\displaystyle{\lim_{k\rightarrow\infty}d_{max}\big(\mathcal{F}^k(\textbf{x}),\tilde{\textbf{x}}\big)=0}.$$ If $(\mathcal{F},X)$ has a globally attracting fixed point we say it is \textit{globally stable}.
\end{definition}

Recall that the constants $L_i$ and $\Lambda_{ij}$ come from (\ref{eq0.1}) and (\ref{eq2.3}) for the local systems $(T,X)$ and interaction $F$ respectively. For the dynamical network $(\mathcal{F},X)$ we define the matrix $M_\mathcal{F}=\Lambda^T\cdot diag[L_1,\dots,L_n]$.  Hence,

$$M_\mathcal{F}=  \left( \begin{array}{cccc}
\Lambda_{11}L_1 & \dots & \Lambda_{n1}L_n \\
\vdots & \ddots & \vdots \\
\Lambda_{1n}L_1 & \dots & \Lambda_{nn}L_n \end{array} \right).$$

Let $\rho(M_\mathcal{F})$ denote the \textit{spectral radius} of the matrix $\Lambda$, i.e. if $\sigma(M_\mathcal{F})$ are the eigenvalues of $\Lambda$ then
$$\rho(M_\mathcal{F})=\max\{|\lambda|:\lambda\in\sigma(M_\mathcal{F})\}.$$

\begin{theorem}\label{stability}
If $\rho\big(M_\mathcal{F}\big)<1$ then the dynamical network $(\mathcal{F},X)$ has a globally attracting fixed point.
\end{theorem}

In the following example we consider a parameterized dynamical network and describe how the network's stability depends on this parameter.

\begin{example}\label{ex2}
Let $T_i:[0,1]\rightarrow [0,1]$ be the map $T_i(x_i)=\sin(\pi x_i)$ for $1\leq i\leq 4$. One can check that the constant $L_i=\pi$ satisfy (\ref{eq0.1}) for each local system $T_i$.

Let $F:[0,1]^4\rightarrow[0,1]^4$ be the interaction given by (\ref{label1}). Using the Lipschitz matrix $\Lambda$ given by (\ref{label2}) the matrix
$$M_{\mathcal{F}}=\left[
\begin{array}{cccc}
\alpha\pi&\alpha\pi&0&\alpha\pi\\
0&0&\alpha\pi&0\\
0&\alpha\pi&\alpha\pi&\alpha\pi\\
\alpha\pi&0&0&0
\end{array}
\right].$$
As one can compute $\rho(M_\mathcal{F})=2\alpha\pi$ implying that the dynamical network $(\mathcal{F},X)$ has a globally attracting fixed point for any parameter value $\alpha<1/2\pi$.
\end{example}

\subsection{Dynamical Network Expansions}
In this section we consider whether it is possible to transform a dynamical network while maintaining the dynamic properties of the network. Our goal is to show (i) that such transformations exist and are analogous to the graph expansion in section 3 and (ii) that these transforms allow for improved estimates on whether the original (untransformed) dynamical network has a unique global attractor.

As defined in section 4 a dynamical network $\mathcal{F}=F\circ T$ is the composition of the network's local dynamics $T$ and the interaction $F$. However, if the system has no local dynamics, i.e. $T=id$ is the identity map, then the dynamical network $\mathcal{F}$ is simply the interaction $F$.

Conversely, any composition $\mathcal{F}=F\circ T$ can be considered to be an interaction. Writing $\mathcal{F}=(F\circ T)\circ id$ the dynamical network $(\mathcal{F},X)$ is simply the interaction $\mathcal{F}=F\circ T$ with no local dynamics.

\begin{remark}\label{remark1}
If $(\mathcal{F},X)$ is considered as a dynamical network with no local dynamics then $\mathcal{M}_\mathcal{F}=\Lambda$ for any constants $\Lambda_{ij}$ satisfying (\ref{eq2.3}) for $\mathcal{F}$. Hence, $(\mathcal{F},X)$ has a globally attracting fixed point if $\rho(\Lambda)<1$.
\end{remark}

Without loss in generality for the remainder of this section we consider $(\mathcal{F},X)$ to be a dynamical network with no local dynamics. That is, the component
$$\mathcal{F}_j:\bigoplus_{i\in\mathcal{I}_j}X_i\rightarrow X \ \ \text{for} \ \ 1\leq j\leq n.$$
Alternatively we write
$$\mathcal{F}_j(\textbf{x}|_{\mathcal{I}_j})=\mathcal{F}_j(x_{j_1},\dots,x_{j_m}), \ \ \text{where} \ \ \mathcal{I}_j=\{j_1,\dots,j_m\}.$$
Note that if we \textit{replace} the variable $x_{j_i}$ of $\mathcal{F}_j$ by the function $f(y_1,\dots,y_k)$ the result is the function
$$\mathcal{F}_j(x_{j_1},\dots,x_{j_{i-1}},f(y_1,\dots,y_k),x_{j_{i+1}},\dots,x_{j_m})$$
having variables $x_{j_1},\dots,x_{j_{i-1}},y_1,\dots,y_k,x_{j_{i+1}},\dots,x_{j_m}$. Additionally, if the sequence $\underline{\gamma}=\ell_1,\dots,\ell_N$ let
$$\mathcal{F}_{j;\underline{\gamma}}(x_{j_1},\dots,x_{j_m})=\mathcal{F}_j(x_{j_1,\underline{\gamma}},\dots,x_{j_m,\underline{\gamma}}).$$
That is, the variables of the function $\mathcal{F}_{j;\underline{\gamma}}$ are indexed by the sequences $j_i,\ell_1,\dots,\ell_N$ for $1\leq i\leq m$.

\begin{definition}
For $G=(V,E,\omega)$ let $S\in st(G)$. The set $S$ is a \textit{complete structural set} of $G$ if\\
(i) each cycle of $G$, including loops, contains a vertex in $S$;\\
(ii) $\omega(e_{ii})\neq\lambda$ for each $v_i\in\bar{S}$; and\\
(iii) each vertex of $V$ belongs to a branch of $\mathcal{B}_S(\Gamma_{\mathcal{F}})$.
\end{definition}

The difference between a structural set and a complete structural set of a graph $G$ is the following. If $S$ is simply a structural set of $G$ then loops of $G$ need not contain a vertex of $S$. However, if $S$ is a complete structural set of $G$ then every cycle of $G$ including loops contains a vertex in $S$. Moreover, each vertex of the graph must belong to some branch of $\mathcal{B}_S(\Gamma_{\mathcal{F}})$.

For $G\in\mathbb{G}$ let $st_0(G)$ denote the set of all complete structural sets of $G$. Also, if the complete structural set $S=\{v_1,\dots,v_m\}$ let $\mathcal{I}_S=\{1,\dots,m\}$ denote the \emph{index set} of $S$.

\begin{definition}\label{admissable}
For $S\in st_0(\Gamma_{\mathcal{F}})$ the set
\begin{equation}\label{eq.add}
\mathcal{A}_S(\mathcal{F})=\{\ell_1,\dots,\ell_N:v_{\ell_1},\dots,v_{\ell_N}\in\mathcal{B}_S(\Gamma_\mathcal{F}), \ N>2\}
\end{equation}
is the set of \textit{admissible sequences} of $\mathcal{F}$ with respect to $S$.
\end{definition}

Let $(\mathcal{F},X)$ be a dynamical network with graph of interactions $\Gamma_\mathcal{F}=(V,E,\omega)$ and suppose $S\in st_0(\Gamma_\mathcal{F})$. For $j\in\mathcal{I}_S$ let $\mathcal{F}_{\left<j,1\right>}$ be the function
$$\mathcal{F}_j=\mathcal{F}_j(x_{j_i},\dots,x_{j_m})$$
in which each variable $x_{j_\ell}$ is replaced by $x_{j_\ell,j}$ if $j_\ell\notin\mathcal{I}_S$.

For $i>1$ let $\mathcal{F}_{\left<j,i\right>}$ be the function
$$\mathcal{F}_{\left<j,i-1\right>}=\mathcal{F}_{\left<j,i-1\right>}(x_{\underline{\gamma}_1},\dots,x_{\underline{\gamma}_t})$$
in which each $x_{\underline{\gamma}_\ell}=x_{\ell_1,\dots,\ell_N}$ is replaced by the function $\mathcal{F}_{\ell_1;\underline{\gamma}_\ell}$ if $\ell_1\notin\mathcal{I}_S$. If $\ell_1\in\mathcal{I}_S$ for each $1\leq\ell\leq t$ then define $(\mathcal{X}_S\mathcal{F})_j=\mathcal{F}_{\left<j,{i-1}\right>}$.

Let $\underline{\gamma}=\ell_1,\dots,\ell_N\in\mathcal{A}_S(\mathcal{F})$. For $1<i<|\underline{\gamma}|=N$ define the $N-2$ spaces
$$X_{i;\underline{\gamma}}=X_{\ell_1}.$$
Additionally, define the functions
$$\mathcal{X}_S\mathcal{F}_{i;\underline{\gamma}}(x_{i-1,\underline{\gamma}})=x_{i-1,\underline{\gamma}}.$$
By way of notation we let
\begin{equation}\label{eq2.4}
X_{N-1,\underline{\gamma}}=X_{\underline{\gamma}}, \  \mathcal{X}_S\mathcal{F}_{N-1,\underline{\gamma}}=\mathcal{X}_S\mathcal{F}_{\underline{\gamma}}, \ \text{and} \ x_{1,\underline{\gamma}}=x_{\ell_1}.
\end{equation}

\begin{definition}\label{expansion}
Suppose $S\in st_0(\Gamma_\mathcal{F})$. Let
$$
\mathcal{X}_S\mathcal{F}=\Big(\bigoplus_{j\in\mathcal{I}_S}\mathcal{X}_S\mathcal{F}_j\Big) \oplus\Big(\bigoplus_{
\begin{smallmatrix}
\underline{\gamma}\in\mathcal{A}_S(\mathcal{F})\\
1<i<|\underline{\gamma}|
\end{smallmatrix}} \mathcal{X}_S\mathcal{F}_{i;\underline{\gamma}}\Big).$$
and
$$X_S=\Big(\bigoplus_{j\in\mathcal{I}_S}X_j\Big)\oplus\Big(\bigoplus_{
\begin{smallmatrix}
\underline{\gamma}\in\mathcal{A}_S(\mathcal{F})\\
1<i<|\underline{\gamma}|
\end{smallmatrix}}X_{i;\underline{\gamma}}\Big).$$
The dynamical network $(\mathcal{X}_S\mathcal{F},X_S)$ is called the \emph{dynamical network expansion} of $(\mathcal{F},X)$ with respect to $S$.
\end{definition}

\begin{figure}
  \begin{center}
    \begin{overpic}[scale=.38]{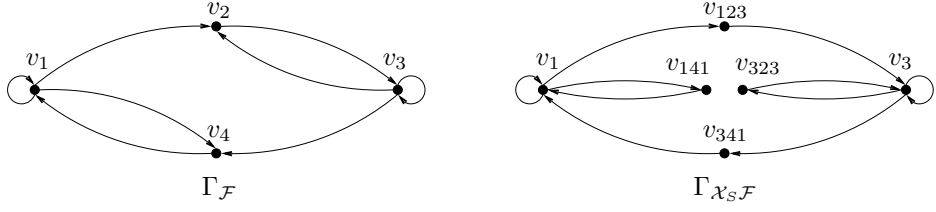}
    \put(2,13){$v_1$}
    \put(21.5,18.5){$v_2$}
    \put(40.5,13){$v_3$}
    \put(21.5,5.25){$v_4$}
    \put(21,-1){$\Gamma_{\mathcal{F}}$}

    \put(57,13){$v_1$}
    \put(95,13){$v_3$}
    \put(71,12.5){$v_{141}$}
    \put(78.5,12.5){$v_{323}$}
    \put(75,5.25){$v_{341}$}
    \put(75,18.5){$v_{123}$}
    \put(74,-1){$\Gamma_{\mathcal{X}_S\mathcal{F}}$}
    \end{overpic}
  \end{center}
\caption{The graph of interactions $\Gamma_\mathcal{F}$ and $\Gamma_{\mathcal{X}_S\mathcal{F}}$ of $(\mathcal{F},X)$ and $(\mathcal{X}_S\mathcal{F},X_S)$ in example \ref{ex13}.}\label{fig2001}
\end{figure}

\begin{example}\label{ex13}
Consider the dynamical network $(\mathcal{F},X)$ where
$$\mathcal{F}(\mathbf{x})=\left[\begin{array}{l}
\mathcal{F}_1(x_1,x_4)\\
\mathcal{F}_2(x_1,x_3)\\
\mathcal{F}_3(x_2,x_3)\\
\mathcal{F}_4(x_1,x_4)
\end{array}
\right].$$
As this network has the same form as the dynamical network in example \ref{ex1} it has the graph of interactions $\Gamma_\mathcal{F}=(V,E,\omega)$ shown in figure \ref{fig2001} (left). Moreover, the set $S=\{v_1,v_3\}$ is a complete structural set of $\Gamma_\mathcal{F}$. To see this note that the cycles of $\Gamma_\mathcal{F}$ form the set
$$\{v_1,v_1; v_3,v_3; v_1,v_4,v_1; v_1,v_2,v_3; v_3,v_4,v_1; v_3,v_2,v_3\}$$
As (i) each cycle of $\Gamma_{\mathcal{F}}$ contains either $v_1$ or $v_3$ and (ii) each vertex in $V$ belongs to the path $v_1,v_2,v_3$ or $v_3,v_4,v_1$ then $S\in st_0(\Gamma_{\mathcal{F}})$. Hence, the expansion $(\mathcal{X}_S\mathcal{F},X)$ is well defined.

To construct this expansion we first consider the components of $\mathcal{F}$ indexed by $\mathcal{I}_S=\{1,3\}$. For $\mathcal{F}_1=\mathcal{F}_1(x_1,x_4)$ note that $x_4$ is indexed by an element not in $\mathcal{I}_S$. Hence, $\mathcal{F}_{\left<1,1\right>}=\mathcal{F}_1(x_1,x_{41})$. Replacing $x_{41}$ by the function $\mathcal{F}_{4;41}=\mathcal{F}_2(x_{141},x_{341})$ yields $\mathcal{F}_{\left<1,2\right>}=\mathcal{F}_1(x_1,\mathcal{F}_4(x_{141},x_{341}))$. Since each variable of $\mathcal{F}_{\left<1,2\right>}$ is indexed by a sequence beginning with an element of $\mathcal{I}_S$ then
$$\mathcal{X}_S\mathcal{F}_1=\mathcal{F}_1(x_1,\mathcal{F}_4(x_{141},x_{341})).$$
Similarly, the function $\mathcal{X}_S\mathcal{F}_3$ can be shown to be
$$\mathcal{X}_S\mathcal{F}_3=\mathcal{F}_3(\mathcal{F}_2(x_{123},x_{323}),x_3).$$

As $\mathcal{A}_S(\mathcal{F})=\{121,123,321,323\}$ then for any $\underline{\gamma}\in\mathcal{A}_S(\mathcal{F})$ there is a single function $\mathcal{X}_S\mathcal{F}_{2,\underline{\gamma}}$ corresponding to $\underline{\gamma}$ given by
$$\mathcal{X}_S\mathcal{F}_{\underline{\gamma}}(x_i)=x_i \ \ \text{for} \ \ \underline{\gamma}=i,2,j.$$ Following definition \ref{expansion} the expansion $\mathcal{X}_S\mathcal{F}$ is given by
\begin{equation}\label{eqnext}
\mathcal{X}_S\mathcal{F}(\textbf{x})=
\left[
\begin{array}{l}
\mathcal{X}_{S}\mathcal{F}_{1}(x_1,x_{141},x_{341})\\
\mathcal{X}_{S}\mathcal{F}_{3}(x_{123},x_{323},x_3)\\
\mathcal{X}_{S}\mathcal{F}_{141}(x_1)\\
\mathcal{X}_{S}\mathcal{F}_{123}(x_1)\\
\mathcal{X}_{S}\mathcal{F}_{341}(x_3)\\
\mathcal{X}_{S}\mathcal{F}_{323}(x_3)\\
\end{array}
\right]=
\left[
\begin{array}{l}
\mathcal{F}_1(x_1,\mathcal{F}_4(x_{141},x_{341}))\\
\mathcal{F}_3(\mathcal{F}_2(x_{123},x_{323}),x_3)\\
x_1\\
x_1\\
x_3\\
x_3\\
\end{array}
\right]
\end{equation}
where $X_S=\big(X_1\oplus X_3\big)\oplus\big(X_{141}\oplus X_{341}\oplus X_{123}\oplus X_{323}\big)$ The graph of interactions $\Gamma_{\mathcal{X}_S\mathcal{F}}$ is shown in figure \ref{fig2001} (right).
\end{example}

Note that the isospectral expansion $\mathcal{X}_S(\Gamma_\mathcal{F})=\Gamma_{\mathcal{X}_S\mathcal{F}}$. This is in fact true in general.

\begin{proposition}\label{expand}
Suppose that $(\mathcal{X}_S\mathcal{F},X_S)$ is a dynamical network expansion of $(\mathcal{F},S)$. Then $\mathcal{X}_S(\Gamma_\mathcal{F})=\Gamma_{\mathcal{X}_S\mathcal{F}}$.
\end{proposition}

A natural question to ask is whether the expansion $(\mathcal{X}_S\mathcal{F},X_S)$ and the initial dynamical network $(\mathcal{F},X)$ have similar dynamics.

\begin{theorem}\label{gafp}
Suppose $({\mathcal{X}_S\mathcal{F}},X_S)$ is a dynamical network expansion of $(\mathcal{F},X)$. If $({\mathcal{X}_S\mathcal{F}},X_S)$ has a globally attracting fixed point then $(\mathcal{F},X)$ has a globally attracting fixed point.
\end{theorem}

The following is an immediate corollary to theorem \ref{gafp} and remark \ref{remark1}.

\begin{corollary}\label{cor30}
Let $({\mathcal{X}_S\mathcal{F}},X_S)$ be a dynamical network expansion of $(\mathcal{F},X)$. Suppose the constants $\tilde{\Lambda}_{ij}$ satisfy (\ref{eq2.3}) for $\mathcal{X}_S\mathcal{F}$. If $\rho(\tilde{\Lambda})<1$ then $(\mathcal{F},X)$ has a globally attracting fixed point.
\end{corollary}

The question then is whether there is any advantage in considering a dynamical network expansion over the original unexpanded network. As it turns out, dynamical network expansions always allow for better estimates (or at least no worse) of a dynamical network's global stability.

\begin{theorem}\label{last} \textbf{(Improved Stability Estimates for Dynamical Networks)}
Let $({\mathcal{X}_S\mathcal{F}},X_S)$ be a dynamical network expansion of $(\mathcal{F},X)$. Suppose there exist constants $\Lambda_{ij}$ satisfying (\ref{eq2.3}) for $\mathcal{F}$. Then there are constants $\tilde{\Lambda}_{ij}$ satisfying (\ref{eq2.3}) for $\mathcal{X}_S\mathcal{F}$ such that $\rho(\tilde{\Lambda})\leq\rho(\Lambda)$.
\end{theorem}

\begin{example}\label{ex1500}
Let $(\mathcal{F},X)$ be the dynamical network given in example \ref{ex2} considered as a network without local dynamics. Hence,
$$\mathcal{F}(\textbf{x};\alpha)=
\left[
\begin{array}{l}
1-\alpha\sin(\pi x_1)\sin(\pi x_4)\\
1-\alpha\sin(\pi x_1)\sin(\pi x_3)\\
1-\alpha\sin(\pi x_2)\sin(\pi x_3)\\
1-\alpha\sin(\pi x_1)\sin(\pi x_3)
\end{array}
\right] \ \ \text{for} \ \ \alpha\in[0,1]$$
where $X=[0,1]^{4}$.

Note that the system $(\mathcal{F},X)$ has the same form as the dynamical network in example \ref{ex13}. Hence, for $S=\{v_1,v_3\}$ the expansion $(\mathcal{X}_S\mathcal{F},X_S)$ given by equation (\ref{eqnext}) is $$\mathcal{X}_S\mathcal{F}(\textbf{x};\alpha)=$$
\begin{equation}\label{nexteq}
\left[
\begin{array}{l}
\mathcal{X}_{S}\mathcal{F}_{1}\\
\mathcal{X}_{S}\mathcal{F}_{3}\\
\mathcal{X}_{S}\mathcal{F}_{141}\\
\mathcal{X}_{S}\mathcal{F}_{123}\\
\mathcal{X}_{S}\mathcal{F}_{341}\\
\mathcal{X}_{S}\mathcal{F}_{323}\\
\end{array}
\right]=
\left[
\begin{array}{l}
1-\alpha\sin(\pi x_1)\sin\big[\alpha\pi (1-\sin(\pi x_{121})\sin(\pi x_{321}))\big]\\
1-\alpha\sin\big[\alpha\pi(1-\sin(\pi x_{123})\sin(\pi x_{323}))\big]\sin(\pi x_3)\\
x_1\\
x_1\\
x_3\\
x_3\\
\end{array}
\right].
\end{equation}

Letting $\tilde{\Lambda}_{ij}=\max_{\textbf{x}\in X_S}|(D\mathcal{X}_S\mathcal{F})_{ji}(\mathbf{x})|$ one can compute that

$$\tilde{\Lambda}=\left[
\begin{array}{cccccc}
\alpha\pi\sin(\alpha\pi) &0&1&1&0&0\\
0&\alpha\pi\sin(\alpha\pi)&0&0&1&1\\
\alpha\pi\sin(\alpha\pi)&0&0&0&0&0\\
0&\alpha\pi\sin(\alpha\pi)&0&0&0&0\\
\alpha\pi\sin(\alpha\pi)&0&0&0&0&0\\
0&\alpha\pi\sin(\alpha\pi)&0&0&0&0
\end{array}
\right] \ \ \text{for} \ \ \alpha\leq1/2$$
where the rows (and columns) of $\tilde{\Lambda}$ are numbered as in equation (\ref{nexteq}). The spectral radius of $\tilde{\Lambda}$ is then given by
$$\rho(\tilde{\Lambda})=\alpha\pi\frac{\sqrt{34-2\cos(2\alpha\pi)}+2\sin(\alpha\pi)}{4}$$
As $\rho(\tilde{\Lambda})<1$ for $\alpha<0.185$ then theorem \ref{last} implies $(\mathcal{F},X)$ has a globally attracting fixed point for any parameter $\alpha<0.185$. Since $1/2\pi<0.185$ then this is better than the estimate gained by direct analysis of $(\mathcal{F},X)$ in example \ref{ex2}.
\end{example}

As a final example we consider a dynamical network of arbitrary size. To determine the stability of this system we require the following. For $A\in\mathbb{W}[\lambda]^{n\times n}$ define the sets
$$\mathcal{E}(A)=\bigcup^n_{i=1}\{\lambda\in \mathbb{C}:|\lambda-A_{ii}|\leq\sum_{j=1,j\neq i}^n| A_{ij}|\}$$

\begin{theorem}{\textbf{(Gershgorin \cite{Gershgorin31})}}\label{gershgorin}
Let $A\in\mathbb{C}^{n\times n}$. Then all eigenvalues of $A$ are contained in the set $\mathcal{E}(A)$.
\end{theorem}

Let $A\in\mathbb{W}[\lambda]^{n\times n}$ be the adjacency matrix of the graph $G$. For $S\in st_0(G)$ we say $S\in st_0(A)$ and let $A_S$ be the adjacency matrix of the reduced graph $\mathcal{R}_S(G)$. The following theorem is an improvement of Gershgorin's result for estimating the nonzero eigenvalues of a complex valued matrix.

\begin{theorem}\label{gersh}
Let $A\in\mathbb{C}^{n\times n}$. If $S\in st_0(A)$ then the nonzero eigenvalues of $A$ are contained in the set $\mathcal{E}(A_S)$. Moreover, $\mathcal{E}(A_S)\subseteq\mathcal{E}(A)$.
\end{theorem}

\begin{proof}
Let $A\in\mathbb{C}^{n\times n}$ be the adjacency matrix of the graph $G$. For $S\in st_0(A)$ suppose $A_S=A_S(\lambda)\in\mathbb{W}[\lambda]^{\ell\times\ell}$. Then equations (\ref{eq0.9}) and (\ref{eq1.0}) imply the $ij$th entry of $A_S(\lambda)$ has the form
$$A_S(\lambda)_{ij}=\sum_{\beta\in\mathcal{B}_{ij}(G,S)}A_{12}\prod_{k=2}^{m-1}\frac{A_{k,k+1}}{\lambda}$$
where the sum is taken over all branches of the form $\beta=v_1,\dots,v_m$. Hence, for any nonzero eigenvalue  $\alpha\in\sigma(A)$ the matrix $A_S(\alpha)\in\mathbb{C}^{\ell\times\ell}$. Moreover, $\alpha\in\sigma(A_S(\lambda))$ by corollary 1 implying $\alpha\in\sigma(A_S(\alpha))$.

By an application of Gershgorin's theorem the inequality
$$|\alpha-A_S(\alpha)_{ii}|\leq\sum_{j=1,j\neq i}^\ell| A_S(\alpha)_{ij}|$$
holds for some $1\leq i\leq \ell$. Hence, $\alpha\in\mathcal{E}(A_S)$.

To verify that the region $\mathcal{E}(A_S)\subseteq\mathcal{E}(A)$ let $S=\{v_{n-\ell-1},v_{n-\ell},\dots,v_n\}$. For $S_k=\{v_{k+1},v_{k+2},\dots,v_n\}$ note that $S_1\in st_0(A)$ and $\mathcal{E}(A_{S_1})=\bigcup_{i=2}^n R_i$ where
$$R_i=\Big\{\lambda\in \mathbb{C}:\Big|\lambda-\Big(A_{ii}+\frac{A_{i1}A_{1i}}{\lambda}\Big)\Big|\leq\sum_{j=2,j\neq i}^n\Big| A_{ij}+\frac{A_{i1}A_{1j}}{\lambda}\Big|\Big\}$$
and $A_{11}=0$. The claim then is $R_i\subset \mathcal{E}_1\cup\mathcal{E}_i$ where
$$\mathcal{E}_i=\{\lambda\in \mathbb{C}:|\lambda-A_{ii}|\leq\sum_{j=1,j\neq i}^n| A_{ij}|\} \ \ \text{for} \ \ 1\leq i\leq n.$$
To see this suppose $\lambda\in R_i$ for fixed $\lambda\in\mathbb{C}$ and $2\leq i\leq n$. Hence,
$$\big|(\lambda-A_{ii})-\frac{A_{i1}A_{1i}}{\lambda}\big|\leq\sum_{j=2,j\neq i}^n\big|A_{ij}+\frac{A_{i1}A_{1j}}{\lambda}\big|.$$
By use of the triangle and reverse triangle inequality then
$$|\lambda-A_{ii}|-\big|\frac{A_{i1}A_{1i}}{\lambda}\big|\leq\sum_{j=2,j\neq i}^n|A_{ij}|+\sum_{j=2,j\neq i}^n\big|\frac{A_{i1}A_{1j}}{\lambda}\big|.$$
Collecting like terms, this implies
$$|\lambda-A_{ii}|+|A_{i1}|\leq\sum_{j=1,j\neq i}^n|A_{ij}|+\sum_{j=2}^n\big|\frac{A_{i1}A_{1j}}{\lambda}\big|.$$

If $\lambda\notin \mathcal{E}_i$ then $|\lambda-A_{ii}|>\sum_{j=1,j\neq i}^n| A_{ij}|$ implying
$$|\lambda-A_{ii}|+|A_{i1}|<|\lambda-A_{ii}|+\frac{|A_{i1}|}{|\lambda|}\sum_{j=2}^n|A_{1j}|.$$
Therefore, $|A_{i1}|\neq 0$ and $|\lambda|<\sum_{j=2}^n|A_{ij}|$. Since $A_{11}=0$ then $\lambda\in\mathcal{E}_1$. This verifies the claim that $R_i\subset \mathcal{E}_1\cup\mathcal{E}_i$ implying $\mathcal{E}(A_{S_1})\subseteq\mathcal{E}(A)$.

Note that each entry of $A_{S_1}(\lambda)_{ij}=A_{i1}A_{1j}/\lambda$ is defined at $\lambda\neq0$. Therefore, the same argument can be used to show $\mathcal{E}((A_{S_1})_{S_2})\subseteq\mathcal{E}(A_{S_1})$. Continuing in this manner it follows that $\mathcal{E}(A_S)\subseteq\mathcal{E}(A)$ since by theorem \ref{theorem-3} the reduced matrix $A_S=(((A_{S_1})\dots)_{S_{n-\ell-2}})_S$.
\end{proof}

Theorem \ref{gersh} can used to estimate the spectral radius associated with an expanded dynamical network. From a computational point of view this can be especially useful when the system size is large as is shown in the following example.

\begin{example}\label{ex14}
Consider the dynamical network $(\mathcal{F},X)$ given by
$$
\mathcal{F}_j(x_{j-1},x_{j+1})=\cos\left(\frac{\pi}{4}x_{j-1}x_{j+1}\right) \ \ 1\leq j\leq 2n
$$
where $X=[0,1]^{2n}$ and $n\geq 2$. Here, the indices are taken$\mod2n$ and the system is assumed to have no local dynamics. Observe that the graph $\Gamma_\mathcal{F}=(V,E,\omega)$ shown in figure 7 (left) for $n=2 $ has a nearest neighbor structure of interactions with periodic boundary conditions.
$$\text{As} \ \ \max_{\textbf{x}\in X}|(D\mathcal{F})_{ji}(\textbf{x})|=
\begin{cases}
\frac{\pi}{4\sqrt{2}} \ \ \text{for} \ \ i=j\pm1\\
0 \ \ \text{otherwise}
\end{cases} \ \ \text{the matrix}$$

\vspace{0.05in}

$$\Lambda=\left[
\begin{array}{cccc}
0 & \frac{\pi}{4\sqrt{2}} & &  \frac{\pi}{4\sqrt{2}}\\
 \frac{\pi}{4\sqrt{2}} & 0 & \ddots &\\
 & \ddots & \ddots &\frac{\pi}{4\sqrt{2}}\\
\frac{\pi}{4\sqrt{2}} & & \frac{\pi}{4\sqrt{2}} & 0\\
\end{array}
\right]$$

\vspace{0.1in}

\noindent is a Lipschitz matrix of $(\mathcal{F},X)$.

Since $\Lambda$ has constant row sums given by $s=\pi/2\sqrt{2}$ then $s$ is an eigenvalue of $\Lambda$ corresponding to the eigenvector $[1 \dots 1]^T\in\mathbb{R}^{2n\times 1}$. However, $s>1$ so theorem \ref{stability} does not directly provide any conclusion about the dynamical network $(\mathcal{F},X)$.

However, $S=\{v_2,v_4,\dots v_{2n}\}$ is a complete structural set of $\Gamma_\mathcal{F}$ since $S$ is a structural set of $\Gamma_\mathcal{F}$ and $\Gamma_\mathcal{F}$ has no loops. Moreover, as
\begin{equation}\label{branch}
\mathcal{B}_S(\Gamma_\mathcal{F})=\{v_i,v_j,v_k: i\in\mathcal{I}_S, \ j=i\pm 1, k=j\pm 1\}
\end{equation}
where each index is taken$\mod 2n$ then each vertex of $\Gamma_\mathcal{F}$ belongs to a branch of $\mathcal{B}_S(\Gamma_\mathcal{F})$.

Note that if $j \in\mathcal{I}_S$ then $j\pm 1\notin\mathcal{I}_S$. Replacing the variables $x_{j-1}$ and $x_{j+1}$ of $\mathcal{F}_j$ by $x_{j-1,j}$ and $x_{j+1,j}$ respectively results in the function
$$\mathcal{F}_{\left<j,1\right>}=\mathcal{F}_j\big(x_{j-1,j},x_{j+1,j}\big).$$

Since the variables $x_{j-1,j}$ and $x_{j+1,j}$ of $\mathcal{F}_{\left<i,1\right>}$ are indexed by sequences beginning with elements not in $\mathcal{I}_S$ then they are replaced by $\mathcal{F}_{j-1;j-1,j}$ and $\mathcal{F}_{j+1;j+1,j}$ respectively to form $\mathcal{F}_{\left<j,2\right>}$. Hence,
$$\mathcal{F}_{\left<j,2\right>}=\mathcal{F}_{j}\big(\mathcal{F}_{j-1}(x_{j-2,j-1,j},x_{j,j-1,j}), \mathcal{F}_{j+1}(x_{j,j+1,j},x_{j+2,j+1,j})\big)$$
As $j\in\mathcal{I}_S$ implies $j\pm2\in\mathcal{I}_S$ then each variable of $\mathcal{F}_{\left<j,2\right>}$ is indexed by a sequence beginning with an element of $\mathcal{I}_S$. This in turn implies that $\mathcal{X}_S\mathcal{F}_j=\mathcal{F}_{\left<j,2\right>}$ for $j\in\mathcal{I}_S$.

\begin{figure}
  \begin{center}
    \begin{overpic}[scale=.45]{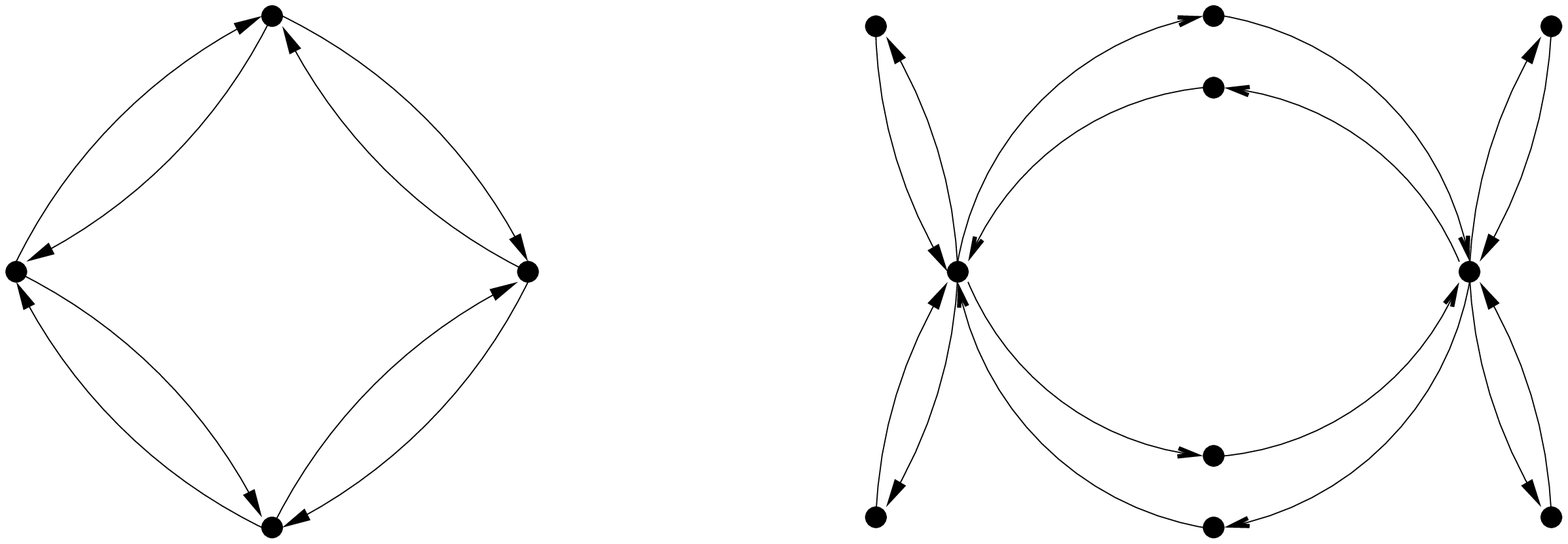}
    \put(16,40){$v_1$}
    \put(-3.5,21){$v_2$}
    \put(35,21){$v_4$}
    \put(16,2.5){$v_3$}
    \put(16,-2){$\Gamma_{\mathcal{F}}$}

    \put(56,21){$v_2$}
    \put(96,21){$v_4$}

    \put(52,39){$v_{212}$}
    \put(52,3){$v_{232}$}
    \put(96,39){$v_{414}$}
    \put(96,3){$v_{434}$}

    \put(75,2.5){$v_{432}$}
    \put(75,12){$v_{234}$}
    \put(75,31){$v_{412}$}
    \put(75,40){$v_{214}$}

    \put(74,-2){$\Gamma_{\mathcal{X}_S\mathcal{F}}$}
    \end{overpic}
  \end{center}
\caption{The graph of interactions $\Gamma_\mathcal{F}$ and $\Gamma_{\mathcal{X}_S\mathcal{F}}$ of $(\mathcal{F},X)$ and $(\mathcal{X}_S\mathcal{F},X_S)$ in example \ref{ex14} for $n=2$.}\label{fig2002}
\end{figure}

As $\mathcal{B}_S(\Gamma_\mathcal{F})$ is given by (\ref{branch}) then
$$\mathcal{A}_S(\mathcal{F})=\{i,j,k:i\in\mathcal{I}_S,j=i\pm1,k=j\pm1\}$$
where each index is taken mod $2n$ and $n\geq 2$. For each $i,j,k=\underline{\gamma}\in\mathcal{A}_S(\mathcal{F})$ there is then a single function corresponding to $\underline{\gamma}$ where
$$\mathcal{X}_S\mathcal{F}_{2,\underline{\gamma}}:X_{1;\underline{\gamma}}\rightarrow X_{2;\underline{\gamma}} \ \ \text{given by} \ \ \mathcal{X}_S\mathcal{F}_{2,\underline{\gamma}}(x_{2;\underline{\gamma}})=x_{2;\underline{\gamma}}.$$ By use of (\ref{eq2.4}) this function can be written as $\mathcal{X}_S\mathcal{F}_{\underline{\gamma}}(x_i)=x_i$.

Following definition \ref{expansion} the dynamical network expansion $\mathcal{X}_S\mathcal{F}:X_S\rightarrow X_S$ is given by

\begin{equation}\label{eq77}
\mathcal{X}_S\mathcal{F}(\mathbf{x})=
\left[
\begin{array}{c}
\mathcal{X}_{S}\mathcal{F}_{2}\\
\vdots\\
\mathcal{X}_{S}\mathcal{F}_{2n}
\end{array}
\right]
\bigoplus
\left[
\begin{array}{c}
\mathcal{X}_S\mathcal{F}_{2n,1,2}\\
\vdots\\
\mathcal{X}_S\mathcal{F}_{2,1,2n}
\end{array}
\right] \ \ \text{where}
\end{equation}

$$\mathcal{X}_S\mathcal{F}_{k,l,m}(x_k)=x_k \ \ \text{for} \ \ k,l,m\in\mathcal{A}_S(\mathcal{F}) \ \ \text{and}$$
$$\mathcal{X}_{S}\mathcal{F}j=\cos\left[\frac{\pi}{4}\cos\left(\frac{\pi}{4}x_{j-2,j-1,j}x_{j,j-1,j}\right) \cos\left(\frac{\pi}{4}x_{j,j+1,j}x_{j+2,j+1,j}\right)\right]$$
for $j\in\{2,4,\dots, 2n\}$. Moreover, the space
$$X_S=\Big(\bigoplus_{j=1}^n X_{2j}\Big)\oplus\Big(\bigoplus_{\underline{\gamma}\in\mathcal{A}_S(\mathcal{F})} X_{\underline{\gamma}}\Big).$$ The graph of interactions $\Gamma_{\mathcal{X}_S\mathcal{F}}$ is shown in figure 7 (right) for n=2.

Letting $\tilde{\Lambda}_{ij}=\max_{\textbf{x}\in X_S}|(D\mathcal{X}_S\mathcal{F})_{ji}(\mathbf{x})|$ one can compute that
$$\tilde{\Lambda}_{ij}=
\begin{cases}
\frac{\pi^2\sin(\pi/4\sqrt{2})}{16\sqrt{2}} \ \ \text{for} \ \ j\in\mathcal{I}_S,\  i\in\mathcal{A}_S(\mathcal{F})\\
\ \ 1 \ \ \ \ \ \ \ \ \ \ \ \ \ \text{for}  \ \ j=k,l,m; \ \ i=k\\
\ \ 0 \ \ \ \ \ \ \ \ \ \ \ \ \ \text{otherwise}
\end{cases}.$$

\begin{figure}
  \begin{center}
    \begin{overpic}[scale=.45]{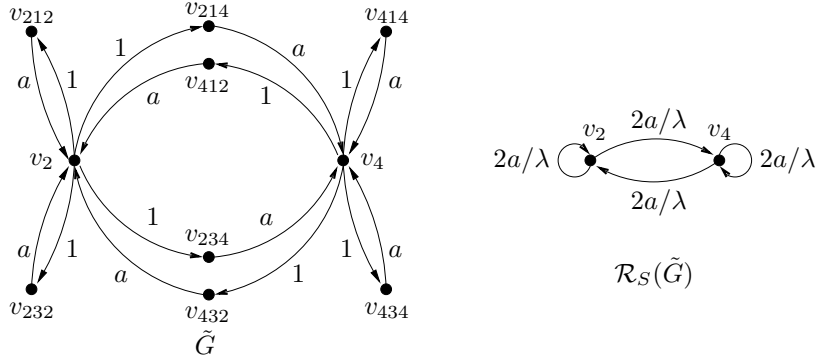}
    \put(0.5,21.5){$v_2$}
    \put(41.5,21.5){$v_4$}
    \put(21,-2){$\tilde{G}$}

    \put(-2,39.5){$v_{212}$}
    \put(-2,3){$v_{232}$}
    \put(42,39.5){$v_{414}$}
    \put(42,3){$v_{434}$}

    \put(19.75,2.5){$v_{432}$}
    \put(19.75,12){$v_{234}$}
    \put(19.75,31){$v_{412}$}
    \put(19.75,40.25){$v_{214}$}

    \put(11,7){$a$}
    \put(33,7){$1$}
    \put(11,35){$1$}
    \put(33,35){$a$}

    \put(15,14){$1$}
    \put(29,14){$a$}
    \put(15,29){$a$}
    \put(29,29){$1$}

    \put(-1,10){$a$}
    \put(5,10){$1$}
    \put(-1,31){$a$}
    \put(5,31){$1$}

    \put(39,10){$1$}
    \put(45,10){$a$}
    \put(39,31){$1$}
    \put(45,31){$a$}

    \put(69,25){$v_2$}
    \put(84.5,25){$v_4$}
    \put(73,7){$\mathcal{R}_S(
    \tilde{G})$}

    \put(75,16){$2a/\lambda$}
    \put(75,26){$2a/\lambda$}
    \put(58,21.25){$2a/\lambda$}
    \put(91,21.25){$2a/\lambda$}

    \end{overpic}
  \end{center}
\caption{The weighted graph $\tilde{G}$ and reduced graph $\mathcal{R}_S(\tilde{G})$ for $n=2$.}\label{fig2002}
\end{figure}

To compute the spectral radius of $\tilde{\Lambda}$ note that if $\tilde{G}$ is the graph with adjacency matrix $\tilde{\Lambda}$ then $S\in st_0(\tilde{G})$, i.e. $S\in st_0(\tilde{\Lambda})$. See figure \ref{fig2002} (left). Moreover, for $n>2$ the $n\times n$ matrix
$$\tilde{\Lambda}_S=\left[
\begin{array}{cccc}
2a/\lambda & a/\lambda & &  a/\lambda\\
a/\lambda & 2a/\lambda & \ddots &\\
 & \ddots & \ddots & a/\lambda\\
a/\lambda & & a/\lambda & 2a/\lambda\\
\end{array}
\right] \ \ \text{where} \ \ a=\frac{\pi^2\sin(\pi/4\sqrt{2})}{16\sqrt{2}}.$$
For $n=2$ the graph $\mathcal{R}_S(\tilde{G})$ is shown in figure \ref{fig2002} (right) having the adjacency matrix
$$\tilde{\Lambda}_S=\left[
\begin{array}{cc}
2a/\lambda & 2a/\lambda\\
2a/\lambda & 2a/\lambda\\
\end{array}
\right].$$

Using theorem \ref{gersh} the nonzero eigenvalues of $\tilde{\Lambda}$ are contained in the region
$$\mathcal{E}(\tilde{\Lambda}_S)=\{\lambda\in\mathbb{C}:|\lambda-2a/\lambda|\leq 2|a/\lambda|\}.$$
It therefore follows that $\rho(\tilde{\Lambda})\leq 2\sqrt{a}\approx.95<1$ implying $\rho(\mathcal{R}_S(\tilde{G}))<1$.

By theorem \ref{maintheorem} then $(\mathcal{X}_S\mathcal{F},X_S)$ has a globally attracting fixed point. Hence, the original unreduced system $(\mathcal{F},X)$ also has this property as well by theorem
\ref{gafp}. Note that if one uses the region $\mathcal{E}(\tilde{\Lambda})$ (i.e. Gershgorin's original theorem \cite{Gershgorin31}) one can only estimate that $\rho(\tilde{\Lambda})\leq 2$. See figure 9.
\end{example}

\begin{figure}
\begin{tabular}{cc}
    \begin{overpic}[scale=.5]{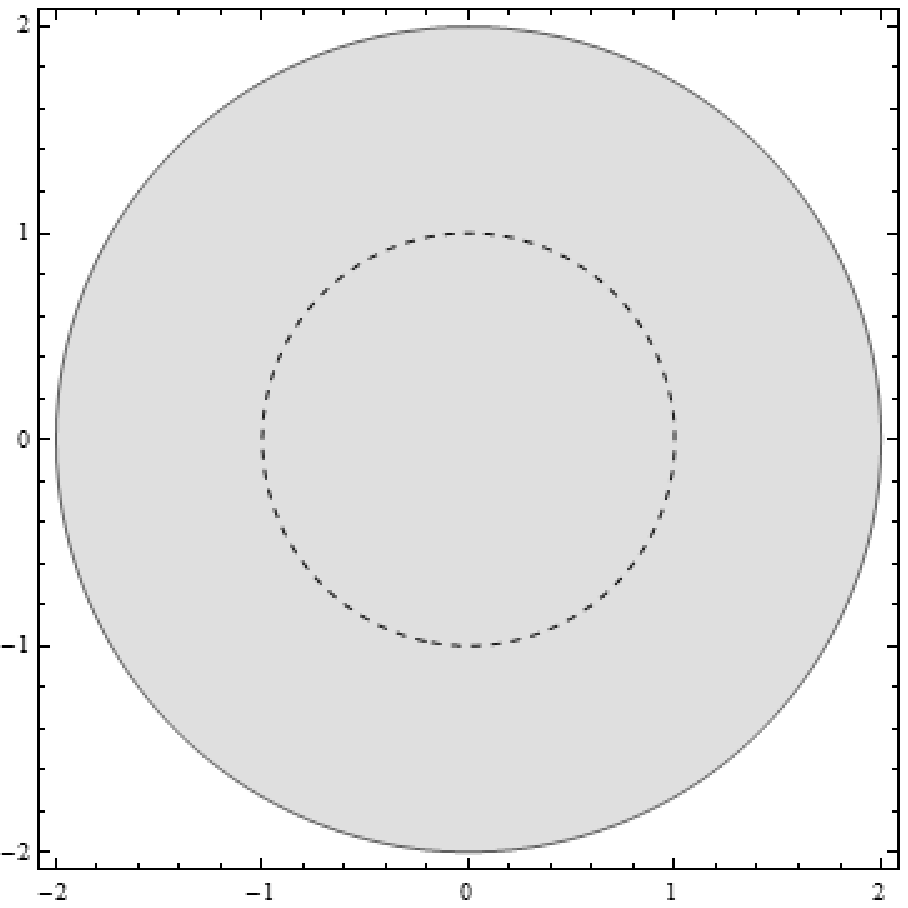}
    \put(44,-8){$\mathcal{E}(\tilde{\Lambda})$}
    \end{overpic} &
    \begin{overpic}[scale=.5]{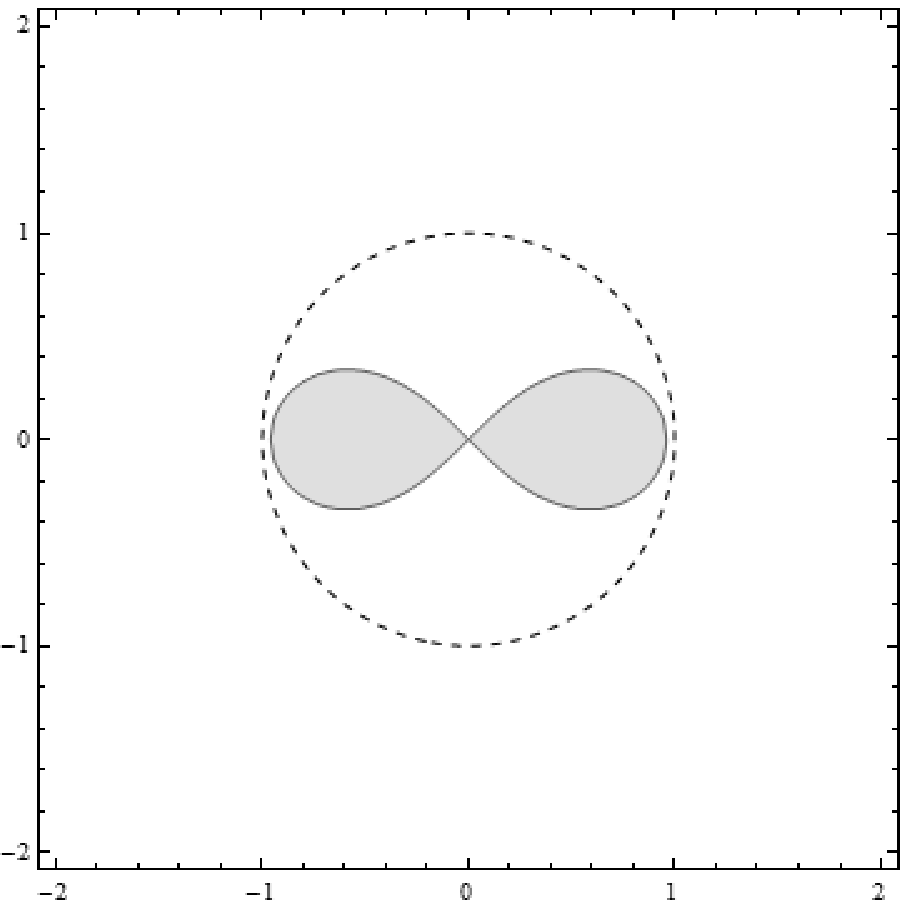}
    \put(44,-8){$\mathcal{E}(\tilde{\Lambda}_S)$}
    \end{overpic}
\end{tabular}
\caption{The regions $\mathcal{E}(\tilde{\Lambda})$ and $\mathcal{E}(\tilde{\Lambda}_S)$ from example 11, where the unit circle is indicated by the dashed lines.}
\end{figure}

We note that the method used in example 11 can be used in general for determining the stability of dynamical networks. To summarize, once a matrix $\tilde{\Lambda}$ corresponding to $(\mathcal{X}_S\mathcal{F},X_S)$ has been found one can use the reduced matrix $\tilde{\Lambda}_S$ to estimate $\rho(\tilde{\Lambda})$ via theorem \ref{gersh}.

Importantly, we note that the method of using network expansions generalizes the standard approach used for determining whether a network has a unique global attractor. As a final observation, in this section, we note that it is possible to sequentially expand a dynamical network and thereby sequentially improve ones estimate of whether the original (untransformed) network has a globally attracting fixed point.

\section{Concluding Remarks}
The large majority of real world networks are dynamic. Yet most studies deal only with the structure (topology) of these networks. As the spectrum is an important dynamical aspect of a network, isospectral graph transformations provide a way of studying the interplay between the topology of a network and its dynamics.

As a tool for investigating networks, isospectral transformations are quite flexible. Isospectral transformations can be used to reduce or expand the size of a network, considered as a weighted graph, while either modifying or maintaining the edge weights of the network. These \emph{isospectral graph reductions} and \emph{isospectral graph expansion} have the following useful properties and applications.

Isospectral graph reductions allow one the ability to uniquely reduce a network to any subsets of its vertex set while preserving the network spectrum. This in turn allows for the introduction of new equivalence relations on the space of all networks where two networks are equivalent if they can be reduced to the same network via some rule.

From the point of view of applications, such rules can be devised by experimentalists to compare different networks modulo some specific network structure. However, this requires the expertise of the experimentalist (biologist, physicist, etc.) to determine an appropriate set of network elements over which to reduce the network. Our procedure, therefore allows the expert the ability to devise and compare different network reductions, i.e. to compare different reduced networks corresponding to various reduction criteria. Such reduction criteria can be designed with respect to vertex or edge centrality, in/out degree, or any other network characteristic that is deemed important.

Isospectral graph expansions can also be used to enlarge a network while preserving its sets of edge weights. As a general multi-dimensional dynamical system has an associated graph structure (i.e. can be considered to be a dynamical network) it is also possible to use this procedure to expand such systems while preserving their dynamics. This can be done in various ways and can, in particular, be used to establish whether the system, e.g. dynamical network, has a globally attracting fixed point.

The new notions, results, and examples presented in this paper demonstrate that the theory of isospectral network transformations is applicable to larger class of networks (e.g. parameter dependent networks) than previously considered. However, we are confident that our approach can be developed even further. For instance, this approach is equally applicable to time-delayed networks (in progress).

\end{document}